\documentclass[12pt,a4paper]{article}
\usepackage{latexsym,amssymb,amsfonts,amsmath,amsthm,nccmath,cases}
\usepackage[usenames,dvipsnames]{xcolor}
\usepackage{enumitem}
\usepackage[hidelinks]{hyperref}
\usepackage[margin=2cm]{geometry}

\numberwithin{equation}{section}
\newtheorem{theorem}{Theorem}[section]
\newtheorem{corollary}[theorem]{Corollary}
\newtheorem{lemma}[theorem]{Lemma}
\newtheorem{proposition}[theorem]{Proposition}

\theoremstyle{definition}
\newtheorem{definition}[theorem]{Definition}
\newtheorem{example}[theorem]{Example}
\newtheorem*{remark}{Remark}
\newtheorem*{open}{Open Problem}

\def \diag {{\rm diag}}
\def \binom#1#2{{#1\choose#2}}

\def \leq {\leqslant}
\def \geq {\geqslant}

\def \mod#1{{\:({\rm mod}\ #1)}}

\let\oldproofname=\proofname
\renewcommand{\proofname}{\rm\bf{\oldproofname}}

\newcommand{\ignore}[1]{}

\title{\bf More nonexistence results\\for symmetric pair coverings}

\author{Nevena Franceti\'{c} \qquad Sarada Herke \qquad Daniel Horsley \\
\\
{\small School of Mathematical Sciences} \\
{\small Monash University} \\
{\small VIC 3800, Australia} \\[0.1cm]
}

\date{}

\begin{document}
\sloppy
\maketitle
\def\baselinestretch{1.2}\small\normalsize


\begin{abstract}
A $(v,k,\lambda)$-covering is a pair $(V, \mathcal{B})$, where $V$ is a $v$-set
of points and $\mathcal{B}$ is a collection of $k$-subsets of $V$ (called
blocks), such that every unordered pair of points in $V$ is contained in at
least $\lambda$ blocks in $\mathcal{B}$. The excess of such a covering is the
multigraph on vertex set $V$ in which the edge between vertices $x$ and $y$ has
multiplicity $r_{xy}-\lambda$, where $r_{xy}$ is the number of blocks which
contain the pair $\{x,y\}$. A covering is symmetric if it has the same number of
blocks as points. Bryant~et~al.~\cite{BBHMS11} adapted the determinant
related arguments used in the proof of the Bruck-Ryser-Chowla theorem to
establish the nonexistence of certain symmetric coverings with $2$-regular
excesses. Here, we adapt the arguments related to rational congruence of
matrices and show that they imply the nonexistence of some cyclic symmetric
coverings and of various symmetric coverings with specified excesses.
\end{abstract}



\section{Introduction}

Suppose $V$ is a set of $v$ \textit{points} and $\mathcal{B}$ is a collection of
$k$-subsets of $V$, called \textit{blocks}.  The pair $(V, \mathcal{B})$ is a
\emph{$(v,k,\lambda)$-design} or a \emph{$(v,k,\lambda)$-covering} if each pair of points of $V$ occurs in exactly $\lambda$ or at least $\lambda$ blocks of $\mathcal{B}$, respectively. The number of blocks in a design is determined by $v$, $k$ and $\lambda$. In the case of coverings, one is usually interested in finding a covering with as few blocks as possible.

It is known that every non-trivial $(v,k,\lambda)$-design has at
least $v$ blocks (see~\cite{Fis}), and consequently designs with exactly $v$
blocks, called
\emph{symmetric designs}, are of particular interest. Many families of symmetric
designs are known to exist, the most famous example being projective planes. One
of the most celebrated results in the study of block designs is the
\emph{Bruck-Ryser-Chowla theorem}~\cite{BruRys,ChoRys} which establishes the
nonexistence of certain symmetric $(v,k,\lambda)$-designs.

The {\em excess} of a $(v,k,\lambda)$-covering $(V,\mathcal{B})$ is the
multigraph with vertex set $V$ in which the multiplicity of the edge joining $x$
and $y$ is $r_{xy}-\lambda$ where $r_{xy}$ is the number of blocks in
$\mathcal{B}$ that contain both $x$ and $y$. For many parameter sets
$(v,k,\lambda)$, a covering with a minimum number of blocks has an
$m$-regular excess, where $m<k-1$. It has been shown in \cite{BC52}
and  \cite{BBHMS11} that, barring some trivial exceptions,
$(v,k,\lambda)$-coverings with fewer blocks than points and 1- or 2-regular
excesses do not exist (see \cite{Horsley} for a recent generalisation). In this
paper, we study symmetric coverings with 2-regular excesses; that is,
coverings with an equal number of points and blocks whose excess is 2-regular.

A $(v,k,\lambda)$-design is a $(v,k,\lambda)$-covering whose excess is
empty, and hence the Bruck-Ryser-Chowla theorem~\cite{BruRys,ChoRys} can be
viewed as establishing the nonexistence of certain symmetric
$(v,k,\lambda)$-coverings with empty excesses. Bose and Connor \cite{BC52} were
able to adapt the arguments used in the proof of the Bruck-Ryser-Chowla theorem
to establish the nonexistence of certain symmetric $(v,k,\lambda)$-coverings
with 1-regular excesses. The case of 2-regular excesses is significantly more
complicated because, for large $v$, there are many non-isomorphic 2-regular
multigraphs on $v$ vertices. Nevertheless, Bryant~et~al.~\cite{BBHMS11} were able
to adapt some of these arguments (those concerning determinants) to the case of
2-regular excesses. In particular, they prove the following result
\begin{theorem} \textup{\cite{BBHMS11}} \label{thm: BBHMS BRC}
Let $v$, $k$, and $\lambda$ be positive integers such that $\lambda \geq 1$ and $3 \leq k < v$. If there exists a
symmetric $(v,k,\lambda)$-covering with $2$-regular excess then
\begin{itemize}
\item when $v$ is even either
\begin{itemize}
\item[$\ast$] $\lambda$ is even and $(v,k,\lambda)=(\lambda+4,\lambda+2,\lambda)$, or
\item[$\ast$] $k-\lambda-2$ is a perfect square and the excess has an odd number of cycles, or
\item[$\ast$] $k-\lambda +2$ is a perfect square and the excess has an even number of cycles;
\end{itemize}
\item when $v$ is odd either
\begin{itemize}
\item[$\ast$] $\lambda$ is odd and $(v,k,\lambda)=(\lambda+4,\lambda+2,\lambda)$, or
\item[$\ast$] the excess has an odd number of cycles.
\end{itemize}
\end{itemize}
\end{theorem}

In \cite{BBHMS11}, Bryant et al. comment that
\begin{quote}
Given the nature of the incidence matrices of the coverings we are considering,
it seems very difficult to adapt the more advanced arguments from the proof of
the Bruck-Ryser-Chowla Theorem.
\end{quote}
Here, we investigate adapting these arguments (which employ Hasse-Minkowski
invariants) to prove the nonexistence of certain symmetric coverings with
$2$-regular excesses.  We establish new
results and also outline some limitations to this approach. Our main findings
are as follows.

\begin{itemize}
	\item In Section \ref{sec: Cp efficiently} we develop an efficient way
to calculate the Hasse-Minkowski invariant for the family of matrices of
interest for this problem (see Lemmas~\ref{lem: full general
cp}~and~\ref{lemma: Bn in terms of gn}).
    \item
In Section \ref{sec: computational results} we present computational results showing that our techniques can be used to rule out the existence of a variety of symmetric coverings with specified excesses. We do not find any parameter sets $(v,k,\lambda)$ for which our techniques completely rule out the existence of a symmetric $(v,k,\lambda)$-covering, but we do find some for which our techniques show there does not exist a cyclic symmetric $(v,k,\lambda)$-covering. This implies the nonexistence of certain interesting almost difference sets.
    \item
In Sections~\ref{sec: p divides a}--\ref{sec: 2 and 3 cycles} we turn our
attention to proving the nonexistence of families of symmetric coverings with
excesses of specific forms. Sections~\ref{sec: p divides a},~\ref{sec: Hamilton cycle excess},~\ref{sec: uniform length cycles} and~\ref{sec: 2 and 3 cycles} deal with coverings for which, respectively, the excess contains an odd number of cycles whose lengths are divisible by 4, the excess is a Hamilton cycle, the excess consists of cycles of uniform length, and the excess consists of only
$2$- and $3$-cycles. In each case we prove a general result and exhibit an infinite family of symmetric coverings with specified excesses whose nonexistence is
established by the result.
\end{itemize}

\section{Preliminaries} \label{sec: overview}

In this section we give an outline of the approach we shall take to establishing the nonexistence of coverings.  We first introduce some notation and concepts that we will require throughout the paper.

If a symmetric $(v,k,\lambda)$-covering has a $2$-regular excess, then by
counting pairs of points we see that $\frac{\lambda v
(v-1)+2v}{2}=\frac{vk(k-1)}{2}$ and hence that $v=\frac{k(k-1)-2}{\lambda}+1$.
The previous equality also implies that each point in such a covering appears
in $\frac{\lambda(v-1)+2}{k-1}=k$ blocks. Conversely, any
$(v,k,\lambda)$-covering with  $v=\frac{k(k-1)-2}{\lambda}+1$ and with a minimum
number of blocks is necessarily symmetric and has $2$-regular excess, provided
that $k \geq 4$.

If the excess of a symmetric covering on $v$ points is $2$-regular, then it is necessarily a vertex-disjoint union of cycles whose lengths add to $v$. Note that here and throughout the paper we consider a pair of parallel edges to form a $2$-cycle.  We say that a 2-regular excess has \emph{cycle type} $[c_1,\ldots,c_t]$ when it is the vertex-disjoint union of cycles of lengths $c_1,\ldots,c_t$ with $c_1 \leq \cdots \leq c_t$.  We say that a cycle type $[c_1,\ldots,c_t]$ is \emph{$v$-feasible} if $c_1 \geq 2$ and $c_1+\cdots+c_t=v$.  Occasionally we will use the shorthand exponential notation $c^\ell$ to represent $c_1,\ldots,c_\ell$ with $c_1=\cdots=c_\ell=c$.

For a prime $p$, each positive integer $n$ can be written uniquely as $\bar{n}p^{\alpha}$ where $\bar{n}$ and $\alpha$ are integers such that $\bar{n} \not\equiv 0 \mod{p}$. We refer to $\bar{n}p^{\alpha}$ as the $p$-factorisation of $n$.

The determinant of a square matrix $X$ is denoted by $|X|$. If $M_1,\ldots,M_t$ are square matrices then we denote by $\diag(M_1,\ldots,M_t)$ the block diagonal matrix with blocks $M_1,\ldots,M_t$. When using this notation we sometimes abbreviate and use $x$ to represent a $1 \times 1$ matrix whose only entry is $x$. For each positive integer $n$, we denote the $n \times n$ identity matrix by $I_n$ and the $n \times n$ all-ones matrix by $J_n$.

Let $(V,\mathcal{B})$ be a $(v,k,\lambda)$-covering (possibly a design) with $b$
blocks and suppose we have ordered the elements of $V$ and $\mathcal{B}$. The
\emph{incidence matrix} $A=(a_{xy})$ of $(V,\mathcal{B})$ is the $v \times b$
matrix such that $a_{xy}=1$ if the $x$th point is in the $y$th block and
$a_{xy}=0$ otherwise. The proof of the Bruck-Ryser-Chowla theorem observes that
if $A$ is the incidence matrix of a symmetric $(v,k,\lambda)$-design, then
$AA^T$ is equal to the $v \times v$ matrix $X=\diag(k,\ldots,k)+\lambda J_v$. It
follows that the determinant of $X$ is a perfect square and also that $X$ is
rationally congruent to the identity matrix (rational congruence is defined
later in this section - see Definition \ref{defn: rationally congruent}).
In~\cite{BruRys}, a contradiction to one of these
facts is obtained for certain parameter sets, thus establishing the nonexistence
of a design with those parameters. Bryant~et~al. \cite{BBHMS11} have adapted
the arguments relating to the determinant of $X$ to symmetric coverings with
2-regular excesses. Here we concentrate on the arguments concerning rational
congruence.

We first establish the structure of the matrix $AA^T$ when $A$ is the incidence 
matrix of a symmetric covering with 2-regular excess.  To do so we will use the 
following family of matrices.
\begin{definition}
For positive integers $a$ and $n$, where $n \geq 2$, we define a matrix $B_n(a)$ as follows.
\begin{displaymath}
B_2(a) = \left( \begin{array}{cc}
		a & 2  \\
  		2 & a
		\end{array} \right), \quad
B_n(a) = \left( \begin{array}{cccccc}
		a & 1 & 0 & 0 & \cdots & 1  \\
  		1 & a & 1 & 0 & \cdots & 0  \\
  		0 & 1 & a & 1 & \cdots & 0  \\
  		\vdots & \vdots & \vdots & \ddots &  & \vdots \\
  		0 & 0 & \cdots & 1 & a & 1  \\
  		1 & 0 & \cdots & 0 & 1 & a
		\end{array} \right) \textrm{ for } n \ge 3.
\end{displaymath}
\end{definition}
Note that the matrices denoted by $B_n(a)$ here were denoted by $B_n'(a)$ in \cite{BBHMS11}. We make the change in order to keep our notation as clean as possible. Further, when there is no risk of confusion, we will sometimes abbreviate $B_n(a)$ to $B_n$.

\begin{definition}\label{defn X}
For positive integers $v$, $k$ and $\lambda$ with $\lambda<k<v$, and any $v$-feasible cycle type $[c_1,\ldots,c_t]$ we define a $v \times v$ matrix
$$X_{(v,k,\lambda)}[c_1,\ldots,c_t]=\diag(B_{c_1}(k-\lambda),B_{c_2}(k-\lambda),\ldots,B_{c_t}(k-\lambda))+\lambda J_v.$$
We sometimes abbreviate $X_{(v,k,\lambda)}[c_1,\ldots,c_t]$ to $X$.
\end{definition}

\begin{proposition}\label{prop: matrix form}
Let $v$, $k$ and $\lambda$ be positive integers such that $\lambda < k < v$. Suppose there exists a symmetric $(v,k,\lambda)$-covering $(V,\mathcal{B})$ whose excess has cycle type $[c_1,\ldots,c_t]$. Then $AA^T=X_{(v,k,\lambda)}[c_1,\ldots,c_t]$, where $A$ is the incidence matrix of $(V,\mathcal{B})$ for some appropriate ordering of $V$ and $\mathcal{B}$. Consequently, $|X_{(v,k,\lambda)}[c_1,\ldots,c_t]|$ is a perfect square.
\end{proposition}
\begin{proof}
Order $V$ so that the first $c_1$ points are the vertices of a $c_1$-cycle in the excess, the next $c_2$ points are the vertices of a $c_2$-cycle in the excess, and so on. Within the vertex set of a cycle, order the points in any way such that consecutive points in the ordering are adjacent in the cycle. Order $\mathcal{B}$ arbitrarily. For $x \in \{1,\ldots,v\}$, the entry in $x$th row and $x$th column of $AA^T$ is the number of blocks that contain the $x$th point, which we have seen is $k$. For distinct $x,y \in \{1,\ldots,v\}$, the entry in $x$th row and $y$th column of $AA^T$ is the number of blocks that contain both the $x$th and $y$th points, which is $\lambda+\mu(xy)$ where $\mu(xy)$ is the multiplicity of the edge $xy$ in the excess. It can now be seen that $AA^T$ has the required form.

Finally, $|X_{(v,k,\lambda)}[c_1,\ldots,c_t]|=|A||A^T|=|A|^2$ where $|A|$ is an integer because $A$ is a $(0,1)$-matrix.
\end{proof}

\begin{definition} \label{defn: rationally congruent}
Two square matrices $X$ and $Y$ of the same size with rational entries are
\emph{rationally congruent}, denoted $X \sim Y$, if there exists an invertible
matrix $P$ with rational entries such that $X = P^T Y P$.
\end{definition}

It is shown in \cite{BBHMS11} that a symmetric $(v,k,\lambda)$-covering with 2-regular excess and $k - \lambda \leq 2$ exists if and only if $(v,k,\lambda)=(\lambda+4,\lambda+2,\lambda)$ (see the proof of Lemma 3.3 of \cite{BBHMS11}). So, in the remainder of this paper, we consider only parameter sets $(v,k,\lambda)$ such that $\lambda+2 < k < v$. Our interest in rational congruence of matrices stems from the following observation.

\begin{proposition}\label{prop: congruent to I}
Let $v$, $k$ and $\lambda$ be positive integers such that $\lambda+2 < k < v$. Suppose there exists a symmetric $(v,k,\lambda)$-covering whose excess has cycle type $[c_1,\ldots,c_t]$.  Then $X_{(v,k,\lambda)}[c_1,\ldots,c_t] \sim I_v$.
\end{proposition}
\begin{proof}
Let $(V,\mathcal{B})$ be such a symmetric $(v,k,\lambda)$-covering.  By Proposition \ref{prop: matrix form},
$$A I_v A^T = AA^T = X_{(v,k,\lambda)}[c_1,\ldots,c_t],$$ where $A$ is the incidence matrix of $(V,\mathcal{B})$ for some appropriate ordering of $V$ and $\mathcal{B}$. Clearly the entries of $A$ are rational. It follows from Lemma 2.1 of \cite{BBHMS11} that $|X_{(v,k,\lambda)}[c_1,\ldots,c_t]| \neq 0$ and hence that $A$ is invertible. So from the definition of rational congruence, $X_{(v,k,\lambda)}[c_1,\ldots,c_t] \sim I_v$.
\end{proof}

To establish that certain matrices $X_{(v,k,\lambda)}[c_1,\ldots,c_t]$ are not
rationally congruent to $I_v$, we employ Hasse-Minkowski invariants. These
are defined in terms of Hilbert symbols which, for our purposes, can be defined
as follows. See  \cite{P45} and \cite[p.~121-122]{Hil98} for proofs that the
definition given here is equivalent to the usual definition. Recall
that  $(\frac{a}{p})$ denotes the well-known Legendre symbol which,
for a prime $p$ and an integer $a$ coprime to $p$, is given by $(\frac{a}{p})=1$
if $a$ is a quadratic residue modulo $p$ and $(\frac{a}{p})=-1$ if $a$ is not a
quadratic residue modulo $p$. We often employ basic properties of the Legendre
symbol (see \cite{AndAnd09} for example).

\begin{definition}
For a prime $p$ and non-zero integers $a$ and $b$ with $p$-factorisations $\bar{a}p^{\alpha}$ and $\bar{b}p^{\beta}$ the \emph{Hilbert symbol} $(a,b)_p$ can be defined by
\begin{numcases}{(a,b)_p=}
\left( \mfrac{-1}{p} \right)^{\alpha \beta}
\left( \mfrac{\bar{a}}{p} \right)^{\beta} \Big( \mfrac{\bar{b}}{p} \Big)^{\alpha}, & if $p>2$; \label{eqn: Hil Legendre p}\\
(-1)^{(\bar{a}-1)(\bar{b}-1)/4 +
\alpha(\bar{b}^2-1)/8 + \beta(\bar{a}^2-1)/8}, & if $p=2$. \label{eqn: Hil Legendre 2}
\end{numcases}
For non-zero integers $a$ and $b$, the \emph{Hilbert symbol} $(a,b)_{\infty}=-1$ is equal to $-1$ if $a$ and $b$ are both negative and  $1$ otherwise.
\end{definition}

From this definition it is easy to deduce some basic facts about Hilbert symbols that we will assume tacitly in the remainder of this paper. For any prime $p$ and non-zero integers $a$, $a'$ and $b$ we have that $(a,b)_p=(b,a)_p$ and $(a,1)_p=1$. Moreover, if $p$ is an odd prime, $a \not\equiv 0 \mod{p}$ and $a \equiv a' \mod{p}$ then $(a,b)_p=(a',b)_p$. Finally, if $a,b \not\equiv 0 \mod{p}$ and $p$ is an odd prime then $(a,b)=1$.

For an $n \times n$ matrix $X$ with rational entries and $i \in \{1,\ldots,n\}$, the \emph{$i$th principal minor of $X$} is the $i \times i$ submatrix of $X$ formed by the entries that are in the first $i$ rows and the first $i$ columns of $X$. We say that $X$ is \emph{nondegenerate} if its $i$th principal minor is invertible for each $i \in \{1,\ldots,n\}$.

\begin{definition} \label{defn: Cp with determinants}
Let $p$ be a prime or $\infty$, let $X$ be an $n \times n$ nondegenerate matrix with rational entries and, for $i \in \{1,\ldots,n\}$, let $X_i$ be the $i$th principal minor of $X$. Then the \emph{Hasse-Minkowski invariant of $X$ with respect to $p$}, denoted $C_p(X)$, is either $+1$ or $-1$ according to
$$C_p(X)=(-1, -|X_n|)_p \medop\prod\limits_{i=1}^{n-1} (|X_i|, -|X_{i+1}|)_p.$$
\end{definition}

For our purposes the critical property of Hasse-Minkowski invariants is as follows.
\begin{theorem}\textup{\cite{P45}} \label{thm: HMIs equal}
Let $X$ and $Y$ be nondegenerate square matrices with rational entries. Then $X \sim Y$ if and only if $C_p(X) = C_p(Y)$ for all primes $p$ and for $p = \infty$.
\end{theorem}

\begin{lemma}\label{lemma: nondegen}
Let $v$, $k$ and $\lambda$ be positive integers such that $\lambda+2 < k < v$, and let $[c_1,\ldots,c_t]$ be a $v$-feasible cycle type.  The matrix $X_{(v,k,\lambda)}[c_1,\ldots,c_t]$ is positive definite and thus nondegenerate.
\end{lemma}

\begin{proof}
Let $X=X_{(v,k,\lambda)}[c_1,\ldots,c_t]$, let
$B=\diag(B_{c_1}(k-\lambda),B_{c_2}(k-\lambda),\ldots,B_{c_t}(k-\lambda))$ and
note that $X=B+\lambda J_v$. Now observe that in each row of $B$, the diagonal
entry is $k-\lambda \geq 3$ and the sum of the absolute values of the
non-diagonal entries is $2$. Thus, by the Gershgorin circle
theorem~\cite{LinAlg}, every eigenvalue of $B$ is positive and hence $B$ is
positive definite. Because $J_v$ is positive semi-definite, $X$ is positive
definite. By Sylvester's criterion~\cite{LinAlg}, this implies that the
determinant of every leading principal minor of $X$ is positive and hence that
$X$ is nondegenerate.
\end{proof}

By Lemma~\ref{lemma: nondegen}, we know that
$C_p(X_{(v,k,\lambda)}[c_1,\ldots,c_t])$ exists for any parameter set such that
$\lambda+2 < k < v$. We shall use this fact tacitly from now on. The
following lemma encapsulates the approach to establishing the
nonexistence of coverings that we shall take in this paper.

\begin{lemma}\label{lemma: main lemma}
Let $v$, $k$ and $\lambda$ be positive integers such that $\lambda+2 < k < v$, and let $[c_1,\ldots,c_t]$ be a $v$-feasible cycle type. There does not exist a symmetric $(v,k,\lambda)$-covering $(V,\mathcal{B})$ whose excess has cycle type $[c_1,\ldots,c_t]$ if either
\begin{align*}
  C_p(X_{(v,k,\lambda)}[c_1,\ldots,c_t]) &= +1 \quad \mbox{for some $p \in \{2,\infty\}$; or} \\
  C_p(X_{(v,k,\lambda)}[c_1,\ldots,c_t]) &= -1 \quad \mbox{for some odd prime $p$}.
\end{align*}
\end{lemma}
\begin{proof}
Suppose that $X_{(v,k,\lambda)}[c_1,\ldots,c_t]$ satisfies the hypotheses of the lemma. It is easy to see (for example, see \cite{BC52}) that the definition of the Hasse-Minkowski invariant implies
$$
C_p(I_v) =  \left\{
 \begin{array}{ll}
 -1, & \textrm{if $p \in \{2,\infty\}$}\\
 +1, & \textrm{if $p$ is an odd prime}.
 \end{array} \right.$$
So $X_{(v,k,\lambda)}[c_1,\ldots,c_t] \nsim I_v$ by
Theorem \ref{thm: HMIs equal}. Thus, by Proposition \ref{prop: congruent to I}
there does not exist a symmetric $(v,k,\lambda)$-covering whose excess has cycle type $[c_1,\ldots,c_t]$.
\end{proof}

We conclude this section with some useful identities involving Hilbert symbols and Hasse-Minkowski invariants, which we shall use frequently in the paper. Sometimes we will not reference them explicitly.  For any non-zero integers $a$, $b$, $s$ and $t$, the following hold. 
\begin{align}
(as^2, bt^2)_p &= (a,b)_p \label{eqn: Hil square} \\
(a_1a_2, b)_p &= (a_1,b)_p (a_2,b)_p \label{eqn: Hil split} \\
(a, -a)_p &= 1 \label{eqn: Hil a -a} \\
(a, a)_p &= (a, -1)_p \label{eqn: Hil a a} \\
(a,b)_p &= (-ab, a+b)_p \label{eqn: Hil a b}
\end{align}

Equations \eqref{eqn: Hil square}--\eqref{eqn: Hil a a} follow easily from our definition of Hilbert symbols (for example, see \cite{H}) and \eqref{eqn: Hil a b} is proved in \cite{BC52}.  The following hold for any nondegenerate $n \times n$ matrix $X$ whose $(n-1)$th
principal minor is denoted $X_{n-1}$ and any nondegenerate $m \times m$ matrix
$Y$ (see~\cite{BC52} for proofs).
\begin{align}
C_p(X) &= C_p(X_{n-1})(|X|, -|X_{n-1}|)_p \label{eqn: cp recursion} \\
C_p(\diag(X,Y)) &= C_p(X) C_p(Y)(-1,-1)_p(|X|,|Y|)_p  \label{eqn: cp block diag}
\end{align}
Moreover, since $(|X|, -|X_{n-1}|)_p = \pm 1$ by the definition of the Hilbert
symbol, \eqref{eqn: cp recursion} can be rearranged into $
C_p(X_{n-1}) = C_p(X) (|X|, -|X_{n-1}|)_p$, which is a form we will often
use.

\section{Computing $C_p(X)$ efficiently} \label{sec: Cp efficiently}

It is possible to calculate the Hasse-Minkowski invariant of a matrix
$X_{(v,k,\lambda)}[c_1,\ldots,c_t]$ directly from Definition~\ref{defn: Cp with
determinants}, but this becomes very slow for large $v$ because
it involves computing the determinant of an $i \times i$ matrix for each $i \in
\{1,\ldots,v\}$. In this section we prove results that allow the Hasse-Minkowski
invariant of matrices $X_{(v,k,\lambda)}[c_1,\ldots,c_t]$ to be efficiently
calculated. These results allow us to perform the computational
investigations in Section \ref{sec: computational results} and they are also
useful in proving our nonexistence results in Sections~\ref{sec: p divides
a}--\ref{sec: 2 and 3 cycles}. We focus on the case where the determinant of our
matrix is a perfect square, because otherwise the corresponding covering cannot
exist by Proposition \ref{prop: matrix form}.

This section is organised as follows. In Lemma \ref{lem: full general cp}, we
show that we can express $C_p(X_{(v,k,\lambda)}[c_1,\dots, c_t])$ in terms of
the parameters of $X_{(v,k,\lambda)}[c_1,\dots, c_t]$ and the Hasse-Minkowski
invariants of the matrices $B_{c_i}(k-\lambda)$. Then we turn our attention to
finding expressions for $C_p(B_{n}(a))$ for positive integers $n$ and $a$. To
aid us in this task we define matrices $B_{n}^*(a)$, which are related to the
matrices $B_{n}(a)$, and a recursive sequence $g_i(a)$ of polynomials in $a$. In
Lemma \ref{lemma: Bn in terms of Bn*} we express $C_p(B_{n}(a))$ in terms of the
Hasse-Minkowski invariant and determinant of $B_{n}^*(a)$ and then in
Lemma~\ref{lemma: Bn* stuff} we express these two values in terms of the
sequence $g_i(a)$. These results allow us to prove Lemma \ref{lemma: Bn in terms
of gn} which gives $C_p(B_{n}(a))$ in terms of the sequence $g_i(a)$. Between
them, Lemmas \ref{lem: full general cp} and \ref{lemma: Bn in terms of gn} give
an efficient method of calculating $C_p(X_{(v,k,\lambda)}[c_1,\dots, c_t])$.

In the proofs of Lemmas \ref{lem: full general cp}, \ref{lemma: Bn in terms of
Bn*} and \ref{lemma: Bn* stuff} we shall make use of the fact that if we perform
a series of elementary row operations on a matrix followed by the corresponding
series of elementary column operations, then the resulting matrix is rationally
congruent to the original matrix. This follows from the definition of rational
congruence because each elementary row operation can be represented as
premultiplication by an elementary matrix $M$ (which has rational entries and is
invertible) and the corresponding elementary column operation can be represented
as postmultiplication by $M^T$. See~\cite{LinAlg} for details on elementary row
operations and elementary matrices.

In \cite{BBHMS11} the determinant of the matrix $B_n$ was found up to a square term.
\begin{lemma} \textup{\cite{BBHMS11}} \label{lemma: BBHMS dn}
Let $n$ and $a$ be positive integers such that $n \geq 2$ and $a > 2$.  Then there exists a polynomial $h \in
\mathbb{Z}[x]$ such that
\[ |B_n(a)| =  \left\{
 \begin{array}{ll}
 (a+2) \cdot h(a)^2, & \textrm{if $n$ is odd},\\
 (a^2-4) \cdot h(a)^2,  & \textrm{if $n$ is even}.\\
 \end{array}
\right. \]
\end{lemma}

\begin{lemma}\label{lem: full general cp}
Let $v$, $k$ and $\lambda$ be positive integers such that $\lambda+2 < k < v$
and let $p$ be a prime or $p=\infty$. Let $[c_1,\ldots,c_t]$ be a $v$-feasible
cycle type and let $e=|\{i:\mbox{$c_i$ is even}\}|$. If $|X_{(v,k,\lambda)}[c_1,\dots, c_t]|$ is a perfect square, then
$$C_p(X_{(v,k,\lambda)}[c_1,\dots, c_t]) = f_p(k-\lambda,\lambda,t,e)\medop\prod\limits_{i=1}^t
C_p(B_{c_i}(k-\lambda))$$
where
$$f_p(a,\lambda,t,e)=(-1,-1)_p^{t-1}(a+2, -1)_p^{{t-e} \choose 2} (a^2-4, -1)_p^{e \choose 2}(a+2, a^2-4)_p^{e(t-e)} (-\lambda, (a+2)^t(a-2)^e)_p.$$
\end{lemma}

\begin{proof}
Let $X=X_{(v,k,\lambda)}[c_1,\dots, c_t]$ and let $X'$ be the matrix $\diag(X,-\lambda)$. Then
\begin{align*}
C_p(X)  &= C_p(X')(|X'|, -|X|)_p && \hbox{by rearranging \eqref{eqn: cp recursion}} \\
&= C_p(X')(-\lambda |X|, -|X|)_p  && \hbox{since $|X'|=-\lambda|X|$} \\
&= C_p(X')(-\lambda, -1)_p  && \hbox{by \eqref{eqn: Hil square} since $|X|$ is square}.
\end{align*}
Let $X''$ be the matrix obtained from $X'$ by adding the last row to all other rows and then adding the last column to all other columns. Note that the $v$th principal minor of $X''$ is $D = \diag(B_{c_1}, B_{c_2}, \dots, B_{c_t})$. Using the equation above, we have
\begin{align*}
C_p(X) &= C_p(X'')(-\lambda, -1)_p && \hbox{since $X'' \sim X'$} \\
 &= C_p(D) (|X''|, -|D|)_p(-\lambda, -1)_p  && \hbox{by \eqref{eqn: cp recursion}} \\
 &= C_p(D) (-\lambda, -|D|)_p(-\lambda, -1)_p && \hbox{since $|X''| = |X'| = -\lambda|X|$} \\
 &= C_p(D)\left(-\lambda, \medop\prod_{i=1}^t |B_{c_i}| \right)_p && \hbox{by \eqref{eqn: Hil split} and the definition of $D$.}
\end{align*}
By repeatedly applying \eqref{eqn: cp block diag}, we have
$$ C_p(D)=\left(\medop\prod_{i=1}^t C_p(B_{c_i}) \right)(-1, -1)_p^{t-1}\left( \,
\medop\prod_{1\leq i < j \leq t} \left(|B_{c_i}|, |B_{c_j}| \right)_p\right).$$
Using Lemma \ref{lemma: BBHMS dn}, (\ref{eqn: Hil a a}) and
(\ref{eqn: Hil split}) we have
\begin{align*}
\medop\prod_{1\leq i < j \leq t} \left(|B_{c_i}|, |B_{c_j}| \right)_p &= (a+2,-1)_p^{{t-e}
\choose 2}  (a^2-4, -1)_p^{e \choose 2} (a+2, a^2-4)_p^{e(t-e)};\hbox{ and} \\
\left(-\lambda, \medop\prod_{i=1}^t |B_{c_i}| \right)_p &= (-\lambda, (a+2)^{t-e})_p
(-\lambda, [(a+2)(a-2)]^e)_p = (-\lambda, (a+2)^t(a-2)^e)_p.
\end{align*}
The result follows by substituting these last three equations into our expression for $C_p(X)$.
\end{proof}

\begin{remark}
When investigating the existence of symmetric $(v,k,\lambda)$-coverings with $2$-regular excesses for some fixed $(v,k,\lambda)$, Lemmas \ref{lemma: main lemma} and \ref{lem: full general cp} can be viewed as operating in the following way. For each $p$, we can (in principle) find the set
$$S_p=\{n \in \{2,\ldots,v\}:C_p(B_n(k-\lambda))=-1\}.$$
By combining Lemmas \ref{lemma: main lemma} and \ref{lem: full general cp}, we can then establish, for a given $t$ and $e$, that any excess of a symmetric $(v,k,\lambda)$-covering that consists of $e$ even cycles and $t-e$ odd cycles either has an even number of cycles with lengths in $S_p$ or has an odd number of cycles with lengths in $S_p$. Which of these two results is established depends on whether $p \in \{2,\infty\}$ and on the value of $f_p(k-\lambda,\lambda,t,e)$ in Lemma \ref{lem: full general cp}. One interesting special case is when $p$ is an odd prime that does not divide $\lambda(a^2-4)$. Then $f_p(k-\lambda,\lambda,t,e)=1$ irrespective of the values of $t$ and $e$ and we can conclude that any excess of a symmetric $(v,k,\lambda)$-covering has an even number of cycles with lengths in $S_p$.
\end{remark}

Next, we introduce a family of sparse square matrices $B^*_n(a)$ which we
later need in the computation of $C_p(B_n(a))$. We will sometimes abbreviate $B^*_n(a)$ to $B^*_n$.
\begin{definition}
For positive integers $a$ and $n$, where $n\geq2$, we define a tridiagonal matrix $B^*_n(a)$ as follows.
\begin{small}
\begin{displaymath}
B^*_2(a) = \left( \begin{array}{cc}
		a-1 & 1  \\
  		1 & a-1
		\end{array} \right), \quad
B^*_n(a) = \left( \begin{array}{cccccc}
		a-1 & 1 & 0 & 0 & \cdots & 0  \\
  		1 & a & 1 & 0 & \cdots & 0  \\
  		0 & 1 & a & 1 & \cdots & 0  \\
  		\vdots & \vdots & \vdots & \ddots &  & \vdots \\
  		0 & 0 & \cdots & 1 & a & 1  \\
  		0 & 0 & \cdots & 0 & 1 & a-1
		\end{array} \right) \textrm{ for } n \ge 3.
\end{displaymath}
\end{small}
\end{definition}

\begin{lemma} \label{lemma: Bn in terms of Bn*}
Let $a$ and $n$ be positive integers such that $a > 2$ and $n \geq 2$, and let $p$ be a prime or $p=\infty$. Then,
$$C_p(B_n(a)) = C_p(B_n^*(a)) (-(a+2)(a-2)^{n+1}, |B_n^*(a)|)_p.$$
\end{lemma}

\begin{proof}
Let $Y'$ be the matrix $\diag(B_n,-1)$.  Then
\begin{align*}
C_p(B_n) &=  C_p(Y')(|Y'|,-|B_n|)_p && \hbox{by rearranging \eqref{eqn: cp recursion}} \\
		&= C_p(Y') (-|B_n|, -1)_p && \hbox{by \eqref{eqn: Hil a a} since $|Y'|=-|B_n|.$}
\end{align*}
Let $Y''$ be the matrix obtained from $Y'$ by adding the last row to the first row and second-last row and then adding the last column to the first column and second-last column. Note that $B_n^*$ is the $n$th principal minor of $Y''$. Using the equation above, we have
\begin{align*}
C_p(B_n) &= C_p(Y'') (-|B_n|, -1)_p && \hbox{since $Y'' \sim Y'$} \\
 &= C_p(B_n^*)(|Y''|,-|B_n^*|)_p(-|B_n|, -1)_p && \hbox{by (\ref{eqn: cp recursion})} \\
 &= C_p(B_n^*)(-|B_n|,|B_n^*|)_p && \hbox{by \eqref{eqn: Hil split} since $|Y''| = |Y'| = -|B_n|$}.
\end{align*}
The result now follows by applying Lemma \ref{lemma: BBHMS dn} and \eqref{eqn: Hil square}.
\end{proof}

\begin{definition}
For each positive integer $n$, let $g_n(a)$ be a polynomial in $a$ defined by the recurrence
\begin{align*}
g_1(a) &= 1; \\
g_2(a) &= a; \\
g_n(a) &= ag_{n-1}(a) - g_{n-2}(a) \quad \hbox{for $n \ge 3$}.
\end{align*}
\end{definition}
Note that $g_n(a)$ is positive for all integers $n \geq1$ and $a \geq 2$. We will sometimes abbreviate $g_n(a)$ to $g_n$.  Below we give $g_n(a)$ for $n \in \{1,\ldots,9\}$.
\begin{center}
\begin{tabular}{ l|l }
 $n$ 	&  $g_n(a)$ \\
 \hline
 1 & $1$ \\
 2 & $a$ \\
 3 & $a^2-1$ \\
 4 & $a^3 -2a$ \\
 5 & $a^4-3a^2+1$ \\
 6 & $a^5 -4a^3+3a$ \\
 7 & $a^6 -5a^4 + 6a^2 -1$\\
 8 & $a^7 -6a^5 +10a^3 - 4a$ \\
 9 & $a^8 -7a^6 +15a^4-10a^2+1$\\
\end{tabular}
\end{center}

\begin{lemma}\label{lemma: Bn* stuff}
Let $a$ be an integer such that $a > 2$ and let $p$ be a prime or $p=\infty$. Then
\begin{itemize}
    \item[(a)]
$|B_n^*(a)|=(a-2)g_n(a)$ for each integer $n \geq 2$;
    \item[(b)]
$C_p(B_n^*(a)) = C_p(B_{n-1}^*(a)) (-g_n(a), g_{n-1}(a))_p$ for each integer $n
 \geq 3 $; and
    \item[(c)]
$C_p(B_n^*(a)) = (-1,2-a)_p \prod\limits_{i=2}^n (-g_i(a), g_{i-1}(a))_p$ for
each integer $n \geq 2$.
\end{itemize}
\end{lemma}

\begin{proof}
\emph{Proof of (a).}  For each positive integer $i \geq 2$, let $T_i$ be the
$i \times i$ tridiagonal matrix such that every entry of the lead diagonal of
$T_i$ is an $a$ and every entry of the superdiagonal and subdiagonal is a
$1$. Using the well-known recursive expression for the determinant of a
tridiagonal matrix (see \cite{Mik03} for example), we see that
$|T_i|=g_{i+1}$ for each positive integer $i$.

When $n=\{2,3\}$, $|B_n^*|$ is easily directly computed. Note that when $n \geq
4$, $B_n^*$ can be obtained from $T_n$ by adding the $n$-dimensional
column vectors $(-1,0,\ldots,0)^T$ and $(0,\ldots,0,-1)^T$ to the first and last
columns, respectively. Thus, using the multilinearity of the determinant as a
function of columns and simplifying, it can be deduced that
\begin{align*}
|B_n^*| &= |T_n| - 2|T_{n-1}| + |T_{n-2}| \\
&= g_{n+1}-2g_{n} +g_{n-1} \\
&= (a-2)g_{n},
\end{align*}
where the last equality follows by substituting $g_{n+1} = ag_{n} - g_{n-1}$. So (a) holds.

\noindent\emph{Proof of (b).} Assume $n \geq 3$ and let $Z'=\diag(B^*_n,-1)$.  Then
\begin{align*}
C_p(B^*_n) &=  C_p(Z')(|Z'|,-|B^*_n|)_p && \hbox{by rearranging \eqref{eqn: cp recursion}} \\
		&= C_p(Z') (-|B^*_n|, -1)_p && \hbox{by \eqref{eqn: Hil a a} since $|Z'|=-|B^*_n|.$}
\end{align*}
Let $Z''$ be the matrix obtained from $Z'$ by adding the last row to the second-last row and third-last row and then adding the last column to the second-last column and third-last column. Note that $Z^{\dag}=\diag(B_{n-1}^*,a-2)$ is the $n$th principal minor of $Z''$. Using the equation above, we have
\begin{align*}
C_p(B_n^*) &= C_p(Z'') (-|B^*_n|, -1)_p && \hbox{since $Z'' \sim Z'$} \\
 &= C_p(Z^{\dag})(|Z''|,-|Z^{\dag}|)_p(-|B^*_n|, -1)_p && \hbox{by \eqref{eqn: cp recursion}} \\
 &= C_p(Z^{\dag})(-|B^*_n|,-(a-2)|B^*_{n-1}|)_p(-|B^*_n|, -1)_p && \hbox{since $|Z''|=|Z'|=-|B^*_n|$} \\
 &= C_p(Z^{\dag})(-|B^*_n|,(a-2)|B^*_{n-1}|)_p && \hbox{by \eqref{eqn: Hil split}} \\
 &= C_p(B_{n-1}^*)(|Z^{\dag}|,-|B_{n-1}^*|)_p(-|B^*_n|,(a-2)|B^*_{n-1}|)_p && \hbox{by (\ref{eqn: cp recursion})} \\
 &= C_p(B_{n-1}^*)((a-2)|B^*_{n-1}|,-|B^*_{n-1}|)_p(-|B^*_n|,(a-2)|B^*_{n-1}|)_p && \hbox{evaluating $|Z^{\dag}|$} \\
 &= C_p(B_{n-1}^*)(|B^*_{n-1}| |B^*_n|, (a-2) |B^*_{n-1}|)_p  && \hbox{by
\eqref{eqn: Hil split}} \\
 &= C_p(B_{n-1}^*) (g_{n-1} g_n, g_{n-1})_p  && \hbox{by part (a) and
\eqref{eqn: Hil square}} \\
 &= C_p(B_{n-1}^*)(-g_n,g_{n-1})_p && \hbox{by \eqref{eqn: Hil split} and
\eqref{eqn: Hil a -a}}.
\end{align*}

\noindent\emph{Proof of (c).} When $n=2$, following the argument used in the
proof of (b) establishes that $C_p(B_2^*)= (-1, 2-a)_p(-g_2,g_1)_p$. Then, when
$n \geq 3$, the statement follows by repeatedly applying (b).
\end{proof}

\begin{lemma}\label{lemma: Bn in terms of gn}
Let $n$ and $a$ be integers such that $a > 2$ and $n \geq 2$, and let $p$ be a prime or $p=\infty$. Then
$$C_p(B_n(a)) = (-(a+2)(a-2)^{n+1},-g_n(a))_p \medop\prod_{i=2}^n (-g_i(a), g_{i-1}(a))_p.$$
\end{lemma}

\begin{proof} Let $\Delta=(a+2)(a-2)^{n+1}$. Combining the results of Lemmas \ref{lemma: Bn in terms of Bn*}(a), \ref{lemma: Bn in terms of Bn*}(c) and \ref{lemma: Bn* stuff} we have
\begin{align*}
C_p(B_n) &= (-1,2-a)_p(-\Delta,(a-2)g_n)_p \textstyle{\prod\limits_{i=2}^n} (-g_i, g_{i-1})_p \\
&= (-1,2-a)_p(-\Delta, 2-a)_p(-\Delta,-g_n)_p \textstyle{\prod\limits_{i=2}^n} (-g_i, g_{i-1})_p && \hbox{by \eqref{eqn: Hil split}}\\
&= (\Delta,2-a)_p(-\Delta,-g_n)_p \textstyle{\prod\limits_{i=2}^n} (-g_i, g_{i-1})_p && \hbox{by \eqref{eqn: Hil split}}.
\end{align*}
The result now follows by observing that \hbox{$(a^2-4,2-a)_p=(a+2,2-a)_p$} by \eqref{eqn: Hil split} and \eqref{eqn: Hil a -a} and furthermore \hbox{$ (a+2,2-a)_p = (-(a+2)(2-a), 4)_p = 1$} by \eqref{eqn: Hil a b} and \eqref{eqn: Hil square}.
\end{proof}

Lemmas~\ref{lem: full general cp}~and~\ref{lemma: Bn in terms of gn} allow us
to compute $C_p(X_{(v,k,\lambda)}[c_1, \dots, c_t])$ for any set of
parameters.  To apply Lemma~\ref{lemma: Bn in terms of gn} we
need to recursively compute the value of $g_i$ for $1 \leq i \leq c_t$ which can be done in linear time in $c_t$. Then we immediately obtain $C_p(X_{(v,k,\lambda)}[c_1,
\dots, c_t])$ as a product of Hilbert symbols.

\textbf{Remark.} Let $v$, $k$ and $\lambda$ be positive integers such that $\lambda+2 < k < v$, and let $[c_1,\ldots,c_t]$ be a $v$-feasible cycle type. Applying Lemma \ref{lemma: main lemma} with $p=\infty$, or with $p$ chosen to be a prime that does not divide any of $\lambda(a^2-4),g_1(k-\lambda),g_2(k-\lambda),\ldots,g_{c_t}(k-\lambda)$, will never rule out the existence of a $(v,k,\lambda)$-covering whose excess has cycle type $[c_1,\ldots,c_t]$. When $p=\infty$, $f_\infty(k-\lambda,\lambda,t,e)=(-1)^{t-1}$ in Lemma \ref{lem: full general cp} and, by Lemma \ref{lemma: Bn in terms of gn}, $C_\infty(B_{n}(k-\lambda)) = -1$ for any $n \in \{2,\dots, v\}$.  When $p$ is a prime that does not divide any of $\lambda(a^2-4),g_1(k-\lambda),g_2(k-\lambda),\ldots,g_{c_t}(k-\lambda)$, $f_p(k-\lambda,\lambda,t,e)=1$ in Lemma \ref{lem: full general cp} and, by Lemma \ref{lemma: Bn in terms of gn}, $C_p(B_{n}(k-\lambda)) = 1$ for any $n \in \{2,\dots, v\}$. So in either case it can be seen from Lemma \ref{lem: full general cp} that Lemma
\ref{lemma: main lemma} tells us nothing. Since the choice $p= \infty$ is never of any use, we do not consider it in the remainder of the paper.

\section{Observations and general computational results} \label{sec: computational results}

We begin this section by noting that it can be seen from Lemma \ref{lem: full
general cp} that for all parameter sets $(v,k,\lambda)$ with $v \equiv 0
\mod{4}$, there will exist cycle types such that Lemma~\ref{lemma: main lemma}
cannot rule out the existence of a $(v,k,\lambda)$-covering whose excess has
that cycle type. To see this consider a $v$-feasible cycle type $[c_1,\ldots,c_t]$ such that $t \equiv 0 \mod{4}$, $|\{i:\hbox{$c_i$ is even}\}| \equiv 0 \mod{4}$ and
$c_{2i-1}=c_{2i}$ for all $i \in \{1,\ldots,\frac{t}{2}\}$ (for example, $[d^{v/d}]$ for any odd divisor $d$ of $v$). From the last of
these conditions it follows that $\prod_{i=1}^tC_p(B_{c_i}(k-\lambda))=1$, and
from the other two conditions we have
that, in Lemma~\ref{lem: full general cp}, $f_p(a,\lambda,t,e)=1$ for any odd
prime $p$ and $f_p(a,\lambda,t,e)=-1$ for $p=2$. As an example, one can take
$(v,k,\lambda)=(36,9,2)$ and cycle types $[3^{12}]$ or $[9^4]$.

In general, we do not expect that there are any parameter sets $(v,k,\lambda)$
for which Lemmas~\ref{lemma: main
lemma},~\ref{lem: full general cp} and~\ref{lemma: Bn in terms of gn} will completely rule out the existence
of a $(v,k,\lambda)$-covering with 2-regular excess. As we shall see
however, for many parameter sets $(v,k,\lambda)$, these results can be
used to establish that many cycle types are not realisable as the excess for a
$(v,k,\lambda)$-covering.  We begin with a small example of this before moving
on to a more general investigation.

\begin{example} \label{ex:11,4,1}
We consider symmetric $(11,4,1)$-coverings. Such a covering necessarily has a $2$-regular excess, and Theorem \ref{thm: BBHMS BRC} implies that this excess has an odd number of cycles. So the possible cycle types of the excess of such a covering are as follows.
$$[11] \quad [2,2,7] \quad [2,3,6] \quad [2,4,5] \quad [3,3,5] \quad [3,4,4] \quad [2,2,2,2,3]$$

Since the values of $g_n(3)$ for $n = 1,\dots,11$ are $1, 3, 8, 21, 55, 144, 377, 987, 2584, 6765, 17711$, respectively, it can be computed that the only choices of $p$ for which Lemma \ref{lemma: main lemma} may rule out $(11,4,1)$-coverings with particular excesses are $2, 3, 5, 7, 11, 13, 17, 19, 29, 41, 47, 89, 199$ (see the remark after Lemma \ref{lemma: Bn in terms of gn}).  For each possible cycle type $[c_1,\ldots,c_t]$ and each choice of $p$ we can apply Lemmas \ref{lem: full general cp} and \ref{lemma: Bn in terms of gn} to
determine $C_p(X_{(11,4,1)}[c_1,\ldots,c_t])$ and then determine whether Lemma
\ref{lemma: main lemma} rules out the existence of a $(11,4,1)$-covering whose
excess has cycle type $[c_1,\ldots,c_t]$.
Below we list for
each cycle type all values of $p$ for which this occurs.
\begin{center}
\begin{tabular}{r||c|c|c|c|c|c|c}
cycle type & $[11]$ & $[2, 2, 7]$ & $[2, 3, 6]$ &  $[2, 4, 5]$ &  $[3, 3, 5]$ &  $[3, 4, 4]$ & $[2, 2, 2, 2, 3]$ \\ \hline
values of $p$ & & $5,13$ & & $3,5$ & & $2, 5$ & $2, 5$ \\
\end{tabular}
\end{center}

Of the three possible cycle types that are not ruled out by Lemma \ref{lemma: 
main lemma}, it transpires that two are realisable and one is not.  A symmetric 
$(11,4,1)$-covering with  Hamilton cycle excess can be constructed from the 
block $[0,1,2,5]$ under the permutation $(0,1,\dots, 10)$, and the following 
list of blocks forms a symmetric $(11,4,1)$-covering with excess cycle type 
$[2,3,6]$.
\begin{center}
$\{0, 1, 5, 8\}, \{0, 1, 6, 9\}, \{0, 1, 7, 10\}, \{0, 2, 3, 4\},  \{1, 2, 3, 4\}, \{2, 5, 6, 7\},$\\
$ \{2, 8, 9, 10\}, \{3, 5, 6, 10\},\{3, 7, 8, 9\}, \{4, 5, 9, 10\}, \{4, 6, 7, 8\}$
\end{center}
We performed an exhaustive computer search to rule out the existence of a symmetric $(11,4,1)$-covering with excess cycle type $[3,3,5]$.
\end{example}

Obviously for other parameter sets we can apply a similar procedure to attempt to rule out the existence of coverings whose excesses have certain cycle types. Our results on symmetric coverings with $\lambda=1$ and $k \in \{4,5,6,7,8,9\}$ are given in Table \ref{table:consequences lambda 1}.  Here it is infeasible to determine the list of relevant values of $p$ as in Example \ref{ex:11,4,1} because finding the prime divisors of $g_v(a)$ becomes increasingly difficult (for example, when $(v,k,\lambda) = (29,6,1)$, $g_{29}(5) \approx 1.18 \times 10^{19}$); instead we test each prime $p < 10^3$.
\begin{table}[h!]
\begin{center}
\begin{tabular}{|l|l|l|l|l|}
\hline
$(v,k,\lambda)$ & number of & number ruled out  & number ruled out by  & number which  \\
         & cycle types & by Theorem \ref{thm: BBHMS BRC}  & Lemma \ref{lemma: main lemma} with $p < 10^3$ & may exist  \\
\hline\hline
$(11,4,1)$ & 14 & 7 & 4 & 3 \\
$(19,5,1)$ & 105 & 52 & 43 & 10\\
$(29,6,1)$ & 847 & 423 & 393 & 31\\
$(41,7,1)$ & 7245  &  3621 & 3376  & 248 \\
$(55,8,1)$ & 65121  &  32555 & 30746  & 1820 \\
$(71,9,1)$ & 609237  &  304604 & 292475  & 12158 \\
\hline
\end{tabular}
\end{center}
\vspace{-0.6cm}

\caption{Consequences of Theorem \ref{thm: BBHMS BRC} and Lemma \ref{lemma: main lemma} for symmetric coverings with $\lambda = 1$ and $k < 10$.}
\label{table:consequences lambda 1}
\end{table}

A \emph{cyclic} symmetric covering is one whose block set can be obtained by applying a cyclic permutation to a single block. A cyclic symmetric $(v,k,\lambda)$-covering with $2$-regular excess is equivalent to a ($v,k,\lambda, v-3)$-almost difference set (see \cite{Now14}).  Such coverings necessarily have excesses consisting of a number (possibly one) of cycles of uniform length.  Table \ref{table: no cyclic covering} lists parameter sets $(v,k,\lambda)$ with $v<200$ for which we can use Lemmas \ref{lem: full general cp}, \ref{lemma: Bn in terms of gn} and \ref{lemma: main lemma}, choosing $p < 10^3$, to computationally rule out the existence of a cyclic symmetric covering.

\begin{table}[h!]
\begin{center}
\footnotesize
\begin{tabular}{|rrrr|rrrr|rrrr|rrrr|}
\hline
$v$ & $k$ & $\lambda$ &  & $v$ & $k$ & $\lambda$  & & $v$ & $k$ & $\lambda$ & & $v$ & $k$ & $\lambda$ &  \\
\hline\hline
153 & 18 & 2 & & 111 & 32 & 9 & & 95 & 49 & 25 & & \textbf{199} & \textbf{98} &
\textbf{48} &  \\
37 & 11 & 3&  & 157 & 38 & 9 & & 53 & 38 & 27&  & 199 & 101 & 51 & \\
169 & 23 & 3 & & \textbf{63} & \textbf{30} & \textbf{14}&  & 81 & 47 & 27 & & 137 & 87 &
55& \\
\textbf{23} & \textbf{10} & \textbf{4} &  & 81 & 34 & 14&  & \textbf{123} & \textbf{60} & \textbf{29}&  & 111
& 79 & 56 & \\
53 & 15 & 4 & & 63 & 33 & 17 & & 123 & 63 & 32 & & 117 & 86 & 63& \\
\textbf{27} & \textbf{12} & \textbf{5} & & 37 & 26 & 18&  & \textbf{135} & \textbf{66} & \textbf{32} & & 157
& 119 & 90&  \\
23 & 13 & 7&  & 121 & 47 & 18&  & 135 & 69 & 35&  & 199 & 134 & 90& \\
161 & 34 & 7 & & 137 & 50 & 18 & & \textbf{171} & \textbf{84} & \textbf{41} & & 161 & 127
& 100& \\
27 & 15 & 8 & & 199 & 65 & 21 & & 171 & 87 & 44 & & 153 & 135 & 119& \\
117 & 31 & 8 & & \textbf{95} & \textbf{46} & \textbf{22}&  & 121 & 74 & 45&  & 169 & 146 &
126& \\
\hline
\end{tabular}
\end{center}
\vspace{-0.6cm}

\caption{Parameter sets $(v,k,\lambda)$ for which Lemma \ref{lemma: main lemma} rules out the existence of a cyclic symmetric covering.}
\label{table: no cyclic covering}
\end{table}

An open problem posed in \cite{Now14} is to find $(v, \frac{v-3}{2}, \frac{v-7}{4},v-3)$-almost difference sets in $\mathbb{Z}_v$ where $v \equiv 3 \mod{4}$ (these are of interest because they produce sequences with desirable autocorrelation properties).
Observe that the parameter sets in boldface in Table \ref{table: no cyclic covering} establish the nonexistence of some $(v,\frac{v-3}{2}, \frac{v-7}{4}, v-3)$-almost difference sets.
Furthermore, using primes $p < 10^3$, we can similarly rule out the existence of $(v,\frac{v-3}{2}, \frac{v-7}{4}, v-3)$-almost difference sets for the following values of $v$, where $v < 800$ (the first eight of which are contained in Table \ref{table: no cyclic covering}).

\begin{center}
23, 27, 63, 95, 123, 135, 171, 199, 207, 215, 231, 243, 255, 267,\\
 271, 307, 343, 351, 355, 363, 367, 371, 375, 399, 407, 411, 471,\\
495, 543, 555, 567, 651, 663, 671, 675, 699, 703, 711, 783
\end{center}

Even when Lemma \ref{lemma: main lemma} does not rule out the existence of a cyclic symmetric covering it may place restrictions on the possible cycle types of the excess.  For example, from Lemmas~\ref{lemma: main lemma},~\ref{lem: full general cp} and~\ref{lemma: Bn in terms of gn}, using primes $p < 10^3$, it follows that if there exists a cyclic symmetric $(51, 24, 11)$-covering, then its excess can only be a Hamilton cycle.  Similarly, if there exists a cyclic symmetric $(75, 36, 17)$-covering, then its excess can only consist of 3 cycles of length $25$.

Below we list the values of $v \equiv 3 \mod{4}$, where $v < 800$, for which we can show (using $p < 10^3$) that if a cyclic symmetric $(v, \frac{v-3}{2}, \frac{v-7}{4})$-covering does exist then it can only have a Hamilton cycle excess.  In the context of almost difference sets, this means the difference repeated $\lambda+1$ times is relatively prime to $v$. We exclude prime values of $v$ from the list since the result is trivial in those cases.
\begin{center}
15, 51, 87, 111, 143, 159, 299, 303, 319, 335, 339, 415, 447,\\
511, 519, 535, 559, 591, 611, 635, 655, 687, 731, 767, 771
\end{center}

In the remaining sections of this paper we address some cases in which
we can prove the nonexistence of coverings whose excesses have certain cycle
types. In Section \ref{sec: p divides a} we show that, for a parameter set
$(v,k,\lambda)$, choosing a value of $p$ dividing $k-\lambda$ allows us to give a
quite general restriction on what cycles types the excesses of symmetric
$(v,k,\lambda)$-coverings may have. In Sections~\ref{sec: Hamilton cycle excess}
and~\ref{sec: uniform length cycles}, we concentrate on the case of coverings
whose excess is a Hamilton cycle or a number of cycles of equal length. These
cases are of particular interest because, as we have seen, any $2$-regular excess of a cyclic
symmetric covering is necessarily of one of these forms.
Finally, in Section \ref{sec: 2 and 3 cycles}, we consider coverings whose
excess is composed of $2$-cycles and $3$-cycles. Results of Bose and Connor (see
\cite{BC52}) already cover the case in which the excess is composed entirely of
$2$-cycles or entirely of $3$-cycles.

\section{Choices of $p$ that divide $k-\lambda$}\label{sec: p divides a}

In this section we obtain a general result on the nonexistence of symmetric
$(v,k,\lambda)$-coverings with certain excesses by choosing values of $p$ which
divide $k-\lambda$. We take advantage of the fact that, under this choice of
$p$, the $p$-factorisations of most of the terms $g_i(k-\lambda)$ are well
behaved.

\begin{theorem}\label{lemma p divides a}
Let $v$, $k$ and $\lambda$ be positive integers such that $k > \lambda + 2$ and
$v = \frac{k(k-1)-2}{\lambda}+1$. For any prime $p \equiv 3 \mod{4}$ such that
$p$ does not divide $\lambda$, and $p$ has odd multiplicity in the prime
factorisation of $k-\lambda$, there does not exist a symmetric
$(v,k,\lambda)$-covering with $2$-regular excess that contains an odd number
of cycles with lengths divisible by $4$ and no cycle of length divisible by
$2p$.
\end{theorem}

\begin{proof}
Let $p \equiv 3 \mod{4}$ be a prime such that $p$ does not divide $\lambda$ and $p$ has odd multiplicity in the prime factorisation of $k-\lambda$. Let $[c_1,\ldots,
c_t]$ be a $v$-feasible cycle type such that $c_1,\ldots,c_t$
contains an odd number of entries divisible by $4$ and contains no
entry divisible by $2p$. Let $X=X_{(v,k,\lambda)}[c_1,\ldots,c_t]$. We may
assume that $|X|$ is a perfect square for otherwise we are finished by
Proposition \ref{prop: matrix form}. By Lemma~\ref{lemma: main lemma}, it
suffices to show that $C_p(X)=-1$.

Let $a=k-\lambda$ and let $\bar{a}p^{\alpha}$ be the $p$-factorisation of $a$. We abbreviate $g_i(a)$ and $B_n(a)$ to $g_i$ and $B_n$ in this proof.
From our hypotheses, $p$ does not divide $\lambda$. Further, since $p$ divides
$a$ it is clear that $p$ divides neither $a+2$ nor $a-2$. Thus, it is easy to
see that, in Lemma~\ref{lem: full general cp}, $f_p(a,\lambda,t,e) = 1$ for all values of $t$ and $e$. So by
Lemma~\ref{lem: full general cp} it suffices to show that, for any integer $n
\geq 2$ which is not divisible by $2p$, $C_p(B_n)=-1$ if and only if $n \equiv
0 \mod{4}$.

Let $n \geq 2$ be an integer that is not divisible by $2p$. From the definition of $g_i$, it is routine to show by induction that, for each non-negative integer $i$,
\begin{equation}\label{eqn gi congruences when p divides a}
g_{2i} \equiv (-1)^{i+1} i \bar{a}p^{\alpha}  \mod{p^{2\alpha}} \quad\mbox{ and
}\quad
g_{2i+1} \equiv (-1)^i \mod{p^{2\alpha}}.
\end{equation}
Observe that $(-(a+2)(a-2)^{n+1},-g_n)_p=1$ for any $n$ because $p$ does not divide $a+2$, $a-2$, or  $g_n$ when $n$ is odd, and
$-(a+2)(a-2) \equiv 2^2 \mod{p}$ when $n$ is even.  From Lemma~\ref{lemma: Bn in
terms of gn},
\begin{equation}\label{eqn Cp product when p divides a}
C_p(B_n) = \medop\prod_{i=2}^n (-g_i, g_{i-1})_p.
\end{equation}
Note that for any positive integer $i$ we have
\begin{align*}
(-g_{2i}, g_{2i-1})_p (-g_{2i+1}, g_{2i})_p &= (-g_{2i}, g_{2i-1})_p (g_{2i-1}, g_{2i})_p && \hbox{since $g_{2i-1} \equiv -g_{2i+1} \mod{p}$ by \eqref{eqn gi congruences when p divides a}} \\
&= (-g_{2i}^2, g_{2i-1})_p && \hbox{by \eqref{eqn: Hil split}} \\
&= (-1, g_{2i-1})_p && \hbox{by \eqref{eqn: Hil square}}\\
& = 1 && \hbox{since $g_{2i-1} \equiv \pm1 \mod{p}$ by \eqref{eqn gi congruences when p divides a}.}
\end{align*}
It follows from \eqref{eqn Cp product when p divides a} that
$$C_p(B_n) =
\left\{
  \begin{array}{ll}
    1, & \hbox{if $n$ is odd;} \\
    (-g_n, g_{n-1})_p, & \hbox{if $n$ is even.}
  \end{array}
\right.$$
For even $n$, it follows from \eqref{eqn gi congruences when p divides a} that $g_{n-1} \equiv (-1)^{n/2-1} \mod{p}$ and, since $n$ is not divisible by $2p$, that $g_{n} = \bar{g_{n}}p^\alpha $ for some integer $\bar{g_{n}}$ not divisible by $p$. Thus, using \eqref{eqn: Hil Legendre p}, for even $n$,
$$(-g_n, g_{n-1})_p = \left(\mfrac{(-1)^{(n/2-1)}}{p}\right)^{\alpha}.$$
It now follows that $C_p(B_n)=-1$ if and only if $n \equiv 0 \mod{4}$ from basic properties of Legendre symbols (note that, from our hypotheses, $p \equiv 3 \mod{4}$ and $\alpha$ is odd).
\end{proof}

It is easy to find infinite families of symmetric coverings with
specified excesses whose existence is ruled out by Theorem \ref{lemma p divides a}. The following corollary, easily proved by setting $p=3$ in Theorem \ref{lemma p divides a}, gives one example.

\begin{corollary} \label{c: infinite family odd number of 4 cycles}
If $k \equiv 7,31,34,58 \mod{72}$ and $v = k(k-1)-1$, then there does not exist a symmetric $(v, k, 1)$-covering with excess having cycle type $[2,3,4^{(v-5)/4}]$.
\end{corollary}

In Table \ref{table: lemma p divides a} we list the parameters $(v,k,\lambda)$
where $1 \le \lambda \le 2$ and $\lambda + 2 < k < 30$ for which there exists
a prime $p \equiv 3 \mod{4}$ that has odd multiplicity in the prime
factorisation of $k-\lambda$.  For each parameter set, we uniformly at random
sample 1000 distinct integer partitions of $v$ which are $v$-feasible
cycle types, or consider all such partitions if $v$
is small. Of the cycle types not forbidden as excesses by Theorem \ref{thm: BBHMS BRC}, we list the proportion which are ruled out using Theorem \ref{lemma p divides a}.

\begin{table}[h!]
\begin{center}
\footnotesize
\begin{tabular}{|lll|l|l||lll|l|l|}
\hline
 $v$ & $k$ & $\lambda$ & $p$ & proportion ruled out & $v$ & $k$ & $\lambda$ & $p$ & proportion ruled out \\\hline\hline
11 & 4 & 1 & 3 & 0.143  & 10 & 5 & 2 & 3 & 0.167\\
41 & 7 & 1 & 3 & 0.206   & 28 & 8 & 2 & 3 & 0.312\\
55 & 8 & 1 & 7 & 0.422   & 36 & 9 & 2 & 7 & 0.392\\
131 & 12 & 1 & 11 & 0.412 & 78 & 13 & 2 & 11 & 0.442\\
155 & 13 & 1 & 3 & 0.0697 & 91 & 14 & 2 & 3 & 0.143\\
209 & 15 & 1 & 7 & 0.264  & 120 & 16 & 2 & 7 & 0.351\\
239 & 16 & 1 & 3 & 0.0336  &136 & 17 & 2 & 3 & 0\\
379 & 20 & 1 & 19 & 0.458  & 210 & 21 & 2 & 19 & 0\\
461 & 22 & 1 & 3 & 0.021   & 253 & 23 & 2 & 3 & 0.0356\\
461 & 22 & 1 & 7 & 0.171   & 253 & 23 & 2 & 7 & 0.285\\
505 & 23 & 1 & 11 & 0.296  & 276 & 24 & 2 & 11 & 0\\
551 & 24 & 1 & 23 & 0.444  & 300 & 25 & 2 & 23 & 0.485\\
599 & 25 & 1 & 3 & 0.01    & 325 & 26 & 2 & 3 & 0.0179\\
755 & 28 & 1 & 3 & 0.00596  & 406 & 29 & 2 & 3 & 0.0207\\
811 & 29 & 1 & 7 & 0.0936 & & & & & \\
\hline
\end{tabular}
\end{center}
\vspace{-0.6cm}

\caption{Proportion of cycle types ruled out by Lemma~\ref{lemma p
divides a} out of those which were not already ruled out by Theorem~\ref{thm:
BBHMS BRC} from a uniform random sample of $v$-feasible cycle types.}
\label{table: lemma p divides a}
\end{table}

\section{Hamilton cycle excesses} \label{sec: Hamilton cycle excess}

In this section we investigate the existence of symmetric
$(v,k,\lambda)$-coverings whose excess is a Hamilton cycle.  We start with
some computational results. We compute $C_p(X_{(v,k,\lambda)}[v])$ for  $1 \leq
\lambda \leq 5$ and $\lambda + 2 < k < 30$ and $p<10^4$. In our search space, there are 18 possible parameter
sets $(v,k,\lambda)$ for a symmetric covering on even number of points $v$; of
these, 12 cases are ruled out by Theorem~\ref{thm: BBHMS BRC} and only 5 are
ruled out by Lemma~\ref{lemma: main lemma}. On the other hand, there are 61
possible parameter sets $(v,k,\lambda)$ where $v$ is odd;  of these, none are
ruled out by Theorem~\ref{thm: BBHMS BRC} and 26 are ruled out by
Lemma~\ref{lemma: main lemma}. Consequently, we focus our attention on the case where $v$ is odd.

Table~\ref{table: Ham cycles excess}
is a summary of parameters for symmetric coverings which cannot have a
Hamilton cycle excess by Lemma~\ref{lemma: main lemma} and which
are not ruled out by Theorem~\ref{thm: BBHMS BRC}. Although there does not appear to be an obvious pattern in the list of primes $p$ which
rule out the existence of coverings with Hamilton cycle excesses, we observe that values of $p$ that are odd and divide $k$ are often effective when $\lambda = 2$; they are marked in boldface. Next, we generalise this pattern to investigate which cases can be ruled out with a prime $p$ that divides
$k-\lambda+2$.

\begin{table}[h!]
\begin{center}
\footnotesize
\begin{tabular}{|lll|l||lll|l|}
\hline
 $v$ & $k$ & $\lambda$ & $p$ & $v$ & $k$ & $\lambda$ & $p$\\\hline\hline
55 &  8 & 1 & 43, 307 &                          37 & 11 & 3 & 73 \\
109 &  11 & 1 &  1307 &                          169 & 23 & 3 & 337, 2027 \\
305 &  18 & 1 &  6709 &                          271 & 29 & 3 & 3793 \\
341 &  19 & 1 &  557, 2417 &                     & & & \\
& & & &                                          23 & 10 & 4 & 229 \\
21 & 7 & 2 & \textbf{7}, 13 &                             53 & 15 & 4 & 317 \\
28 & 8 & 2 & 2, 3 &                              116 & 22 & 4 & 173, 347 \\
45 & 10 & 2 & 29, 149 &                          127 & 23 & 4 & 1777 \\
55 & 11 & 2 & \textbf{11}, 109, 197 &                     & & & \\
78 & 13 & 2 & 2, 5 &                             27 & 12 & 5 & 2, 3, 107 \\
91 & 14 & 2 & \textbf{7}, 223 &                           93 & 22 & 5 & 991 \\
105 & 15 & 2 & 59, 419, 509 &                    111 & 24 & 5 & 2, 3 \\
153 & 18 & 2 & 5, 71, 101, 2447, 5303 &          141 & 27 & 5 & 281\\
171 & 19 & 2 & \textbf{19}, 113, 227, 1367, 4217, 5813 &  163 & 29 & 5 & 2281\\
190 & 20 & 2 & 37, 113, 797 &&&& \\
231 & 22 & 2 & \textbf{11}, 41 &&&& \\
253 & 23 & 2 & \textbf{23}, 43 &&&& \\
325 & 26 & 2 & 19, 29, 4549 &&&& \\
351 & 27 & 2 & 2, \textbf{3}, 71, 233, 1637 &&&& \\
406 & 29 & 2 & 41, 461 &&&& \\ \hline
\end{tabular}
\end{center}
\vspace{-0.6cm}

\caption{Parameter sets $(v,k,\lambda)$ for which Lemma \ref{lemma: main lemma} rules out the existence of a symmertic covering with Hamilton cycle excess.}
\label{table: Ham cycles excess}
\end{table}

The remainder of this section is organised as follows. For choices of $p$ that 
divide $a+2$, Lemmas \ref{lemma g lambda 2} and \ref{lemma g star deltas} give 
results about the behaviour of $g_i(a)$ modulo $p$ and Lemma \ref{lemma lambda 2 
Ham cycle} finds an expression for $C_p(B_n(a))$. Lemmas \ref{lemma g star 
deltas} and \ref{lemma lambda 2 Ham cycle} will also be used in Section 
\ref{sec: uniform length cycles}. We then use these results, along with the 
technical Lemma \ref{lemma v mod p}, to prove Theorem \ref{theorem: no hamilton 
cycle v odd} which establishes the nonexistence of a symmetric 
$(v,k,\lambda)$-covering with Hamilton cycle excess for an infinite number of 
parameter sets.

\begin{lemma} \label{lemma g lambda 2}
If $p$ is a prime and $a$, $s$ and $\alpha$ are positive integers such that $a+2 = p^{\alpha} s$,
then $g_n(a) \equiv (-1)^{n+1} n \mod{p^{\alpha}}$ for each positive integer $n$.
\end{lemma}
\begin{proof} Obviously $a \equiv -2 \mod{p^{\alpha}}$. Using this and the recursive definition of $g_i$, the result follows easily by induction.
\end{proof}

\begin{lemma} \label{lemma g star deltas}
Let $p$ be an odd prime and $a$ and $n$ be positive integers such that $a > 2$, $a+2 \equiv 0 \mod{p}$, and $n \equiv 0 \mod{p}$. Let $sp^{\alpha}$ and $\bar{n}p^\delta$ be the $p$-factorisations of $a+2$ and $n$ respectively. If $(p,\alpha) \neq (3,1)$, then $g_n(a) = \bar{g}p^{\delta}$ for some integer $\bar{g} \equiv (-1)^{n+1} \bar{n} \mod {p}$.
\end{lemma}
\begin{proof}
We show that $\frac{g_n}{p^\delta}$ is an integer congruent to $(-1)^{n+1}
\bar{n}$ modulo $p$ which will suffice to prove the result.  The value of $g_n$
is defined by a second order recurrence relation. Solving this, we see that $g_n
= \frac{\zeta_1^{n} - \zeta_2^{n}}{\zeta_1 - \zeta_2}$
where $\zeta_1 = \frac{1}{2}(a + \sqrt{a^2-4})$, $\zeta_2 = \frac{1}{2}(a -
\sqrt{a^2-4})$. Let $b = \sqrt{a^2-4}$. Now,
\begin{align*}
g_n &= \mfrac{1}{2^n b} \left( (a+b)^n - (a-b)^n \right) \\
	&= \mfrac{1}{2^n b} \left( \medop\sum\limits_{i=0}^n \mbinom{n}{i}a^{n-i}b^i -
\medop\sum\limits_{i=0}^n \mbinom{n}{i}a^{n-i} (-1)^i b^i \right) \\
	&= \mfrac{1}{2^n b} \left( 2 \medop\sum\limits_{i=0}^{\lfloor (n-1)/2
\rfloor} \mbinom{n}{2i+1} a^{n-2i-1} b^{2i+1} \right) \\
	&= \mfrac{1}{2^{n-1}} \medop\sum\limits_{i=0}^{\lfloor (n-1)/2
\rfloor} \mbinom{n}{2i+1} a^{n-2i-1} (a^2-4)^{i} \\
    &= \mfrac{1}{2^{n-1}} \medop\sum\limits_{i=0}^{\lfloor (n-1)/2
\rfloor} T_i   \label{eqn gn terms
in binomial coeff}
\end{align*}
where, for each $i \in \{0,\ldots,\lfloor\frac{n-1}{2}\rfloor\}$,
$$T_i = \mbinom{n}{2i+1} (sp^\alpha - 2)^{n-2i-1} sp^{\alpha i}
(sp^\alpha -4)^i. $$
Since $n = \bar{n}p^\delta$, it is clear that $T_0$ is divisible by
$p^\delta$. For each $i \in \{1,\ldots,\lfloor\frac{n-1}{2}\rfloor\}$, we will show that $T_i$ is divisible by $p^{\delta+1}$. 

Fix $j \in \{1,\ldots,\lfloor\frac{n-1}{2}\rfloor\}$ and let $\bar{m}p^\xi$ be the $p$-factorisation of $2j+1$. Since $j,\alpha,\bar{m} \geq 1$, $(p,\alpha) \neq (3,1)$ and $j =
\frac{\bar{m}p^\xi-1}{2}$, it is not difficult to see that $\alpha j \geq j >
\xi$. Note that $T_j$ is divisible by $p^{\alpha j}$. If $\xi>\delta$, then it 
can be seen that $\alpha j > \xi \geq
\delta+1$ and hence that $T_j$ is divisible by $p^{\delta+1}$. If $\xi \leq \delta$, then $\binom{n}{2j+1} =
\frac{n}{2j+1}\binom{n-1}{2j} = p^{\delta -
\xi}\frac{\bar{n}}{\bar{m}}\binom{n-1}{2j}$, and so $\binom{n}{2j+1}$ is
divisible by $p^{\delta - \xi}$. So $T_j$ is divisible by $p^{\alpha j + \delta
- \xi}$, $\alpha j \geq \xi+1$, and $T_j$ is divisible by $p^{\delta+1}$.

So $T_i$ is divisible by $p^{\delta +1}$ for each $i \in \{1,\ldots,\lfloor\frac{n-1}{2}\rfloor\}$ and $T_0$ is divisible by $p^\delta$. It follows that $\frac{g_n}{p^\delta}$ is an integer and
\begin{align*}
\frac{g_n}{p^\delta} &\equiv \mfrac{1}{2^{n-1}}\mfrac{T_0}{p^\delta} \mod {p} \\
 &\equiv  \mfrac{1}{2^{n-1}} \bar{n}(sp^\alpha-2)^{n-1} \mod {p} \\
 &\equiv  \mfrac{1}{2^{n-1}} \bar{n}(-2)^{n-1} \mod {p} \\
 &\equiv   (-1)^{n+1} \bar{n} \mod{p}.
\end{align*}
The result follows.
\end{proof}

\begin{lemma} \label{lemma lambda 2 Ham cycle}
Let $p$ be an odd prime and $a$ and $n$ be positive integers such that $a > 2$ and $a +2 \equiv 0 \mod{p}$. Then
\begin{itemize}
    \item[(i)]
$C_{p}(B_n(a)) = \left( \mfrac{(-1)^{\alpha\beta + \alpha +
\beta} s^\beta \bar{g}^\alpha}{p} \right)$; and
    \item[(ii)]
$C_{p}(B_n(a)) = \left( \mfrac{(-1)^{n}n}{p} \right)^\alpha$ when $n \not\equiv 0 \mod{p}$;
\end{itemize}
where $sp^{\alpha}$ and $\bar{g}p^{\beta}$ are the $p$-factorisations of $a+2$ and $g_n(a)$ respectively.
\end{lemma}

\begin{proof}
Using Lemma \ref{lemma g lambda 2}, it is not difficult to see that (i) implies (ii), so it suffices to show that (i) holds. From Lemma \ref{lemma: Bn in terms of gn}, we have
\begin{equation} \label{eqn CPBn full product}
C_p(B_n) = (-(a+2)(a-2)^{n+1},-g_n)_p \prod\limits_{i=2}^n (-g_i, g_{i-1})_p.
\end{equation}
Next, we find an expression for  $\prod_{i=2}^n (-g_i, g_{i-1})_p$.

If neither $g_i$ nor $g_{i-1}$ is divisible by $p$, then $(-g_i, g_{i-1})_{p} =
1$.  Thus, by Lemma~\ref{lemma g lambda 2},
$$\prod_{i=2}^n (-g_i, g_{i-1})_p=\prod_{i\in S} (-g_i, g_{i-1})_p,$$
where $S=\{i \in \{2,\ldots,n\}: i\equiv 0,1 \mod{p}\}$.

For each integer $j \equiv 0 \mod{p}$, let $\bar{g}_jp^{\beta_j}$ be the $p$-factorisation of $g_{j}$. Note that $\beta_n=\beta$. For each integer $j \equiv 0 \mod{p}$, it
can be seen using Lemma~\ref{lemma g lambda 2} that both $-g_{j+1}$ and
$g_{j-1}$ are congruent to $(-1)^{j+1}$ modulo $p$ and hence, by \eqref{eqn: Hil
Legendre p}, we have
\[
(-g_{j+1}, g_j)_{p} = (-g_j,g_{j-1})_{p} = \left( \mfrac{(-1)^{j+1} }{p} \right)^{\beta_j}.
\]
Obviously this implies that $(-g_{j+1}, g_j)_{p}(-g_j,g_{j-1})_{p}=1$ for each integer $j \equiv 0 \mod{p}$. Using these facts it can be seen that
\begin{equation} \label{eqn gi product formula}
\prod_{i=2}^n (-g_i, g_{i-1})_p=
\left\{
  \begin{array}{ll}
    1, & \hbox{if $n \not\equiv 0 \mod{p}$;} \\
    (-g_n, g_{n-1})_p=\left( \mfrac{(-1)^{n+1} }{p} \right)^{\beta}, & \hbox{if $n \equiv 0 \mod{p}$.}
  \end{array}
\right.
\end{equation}
The proof now splits into cases according to whether $n$ is odd or even.

\noindent {\bf Case 1.} Suppose that $n$ is odd. Then, by \eqref{eqn gi product formula}, $\prod_{i=2}^n (-g_i, g_{i-1})_p=1$. So, from \eqref{eqn CPBn full product},
\begin{align*}
  C_p(B_n) &= (-a-2,-g_n)_p && \hbox{using \eqref{eqn: Hil square}} \\
  &= (-sp^{\alpha},-\bar{g}p^{\beta})_p  \\
  &= (s,\bar{g}p^{\beta})_p(-p^\alpha,\bar{g})_p(-p^\alpha,-p^{\beta})_p && \hbox{by \eqref{eqn: Hil split}} \\
  &= \left(\mfrac{s}{p} \right)^\beta
\left(\mfrac{\bar{g}}{p}\right)^{\alpha}  \left(\mfrac{-1}{p} \right)^{\alpha\beta+\alpha+\beta} && \hbox{by \eqref{eqn: Hil Legendre p}}.
\end{align*}
Using basic properties of Legendre symbols, the result follows.

\noindent {\bf Case 2.} Suppose that $n$ is even. Then, by \eqref{eqn gi product formula}, $\prod_{i=2}^n (-g_i, g_{i-1})_p=(\frac{-1}{p})^{\beta}$.
So, from \eqref{eqn CPBn full product},
\begin{align*}
  C_p(B_n) &= ((a+2)(2-a),-g_n)_p\left(\mfrac{-1}{p}\right)^{\beta} && \hbox{using \eqref{eqn: Hil square}} \\
  &= (sp^{\alpha}(4-sp^{\alpha}),-\bar{g}p^{\beta})_p\left(\mfrac{-1}{p}\right)^{\beta} \\
  &=
(s(4-sp^{\alpha}),-\bar{g}p^{\beta})_p(p^{\alpha}, \bar{g})_p(p^{\alpha},-p^{
\beta})_p\left(\mfrac{-1}{p}\right)^{\beta} && \hbox{by \eqref{eqn: Hil
split}}\\
  &= \left(\mfrac{4s}{p} \right)^\beta
\left(\mfrac{\bar{g}}{p}\right)^{\alpha}  \left(\mfrac{-1}{p} \right)^{\alpha\beta+\alpha}\left(\mfrac{-1}{p}\right)^{\beta} && \hbox{by \eqref{eqn: Hil Legendre p}}.
\end{align*}
Using basic properties of Legendre symbols, the result follows (note that $4=2^2$).
\end{proof}

\begin{lemma} \label{lemma v mod p}
Let $v$, $k$ and $\lambda$ be positive integers such that $k > \lambda + 2$ and $v = \frac{k(k-1)-2}{\lambda}+1$, and let $p$ be an odd prime such that $k-\lambda+2 \equiv 0 \mod{p}$.  Then
\begin{itemize}
    \item[(i)]
$\lambda v = (\lambda -2)^2 + sp^{\alpha}(sp^{\alpha}+2\lambda -5)$ where $sp^{\alpha}$ is the $p$-factorisation of $k-\lambda+2$; and
    \item[(ii)]
$v \equiv 0 \mod {p}$ if and only if $\lambda \equiv 2 \mod {p}$.
\end{itemize}
\end{lemma}

\begin{proof}
Note that $\alpha \geq 1$. Because $v = \frac{k(k-1)-2}{\lambda}+1$ and $k = sp^{\alpha}+\lambda -2$, a straightforward calculation yields (i).

Suppose that $\lambda \equiv 2 \mod {p}$. Then $p$ divides $\lambda - 2$ and $p$ does not divide $\lambda$. So it follows from (i) that $v \equiv 0 \mod {p}$. Now suppose that $v \equiv 0 \mod {p}$. Then it follows immediately from (i) that $p$ divides $(\lambda -2)^2$ and hence that $\lambda \equiv 2 \mod {p}$.
\end{proof}

\begin{theorem}\label{theorem: no hamilton cycle v odd} Let $v$, $k$ and $\lambda$ be positive integers such that $k>\lambda+2$, $v = \frac{k(k-1)-2}{\lambda}+1$ and $v$ is odd. There does not exist a symmetric $(v,k,\lambda)$-covering whose excess is a Hamilton cycle if there is a prime $p \equiv 3 \mod{4}$ that divides both $v$ and $k-\lambda+2$ and such that either
\begin{itemize}
    \item[(i)] $\alpha$ is odd, $(p, \alpha) \neq
(3, 1)$, and either $\lambda>2$ and $\alpha < 2\gamma$ or $\lambda=2$; or
	\item[(ii)] $\lambda>2$, $\alpha = 2\gamma$, and $\delta$ is
odd;
\end{itemize}
where $sp^{\alpha}$ and $\bar{v}p^\delta$ are the $p$-factorisations of
$k-\lambda+2$ and $v$ respectively, and $\ell p^{\gamma}$ is the $p$-factorisation of $\lambda - 2$ if $\lambda >2$.
\end{theorem}

\begin{proof} Let $p$ be a prime satisfying the hypotheses of the lemma and let $sp^{\alpha}$ and $\bar{v}p^\delta$ be the $p$-factorisations of $k-\lambda+2$ and $v$ respectively. Note that $\alpha,\delta \geq 1$. Let $X=X_{(v,k,\lambda)}[v]$. We may assume that $|X|$ is a perfect square for otherwise we are finished by Proposition \ref{prop: matrix form}. By Lemma \ref{lemma: main lemma} it suffices to show that $C_{p}(X) = -1$.

By Lemma \ref{lem: full general cp}, remembering that $v$ is odd, we have
\begin{align*}
  C_p(X) &= C_p(B_v(k-\lambda))(-\lambda,k-\lambda+2)_p \\
  &= C_p(B_v(k-\lambda))(-\lambda,sp^{\alpha})_p \\
  &= C_p(B_v(k-\lambda))\left(\mfrac{-\lambda}{p}\right)^{\alpha},
\end{align*}
where we used \eqref{eqn: Hil Legendre p} and the fact that $p$ does not divide
$\lambda$ to deduce the last equality. By Lemma~\ref{lemma g
star
deltas}, $g_v(k-\lambda) = \bar{g}p^{\delta}$ for some integer $\bar{g} \equiv
\bar{v} \mod {p}$ and thus, by Lemma~\ref{lemma lambda 2 Ham cycle}(i) (noting that $(p,\alpha) \neq (3,1)$), we have
\begin{equation}\label{eq no hamilton cycle v odd}
 C_{p}(X) = \left( \mfrac{(-1)^{\alpha\delta + \alpha +
\delta} s^\delta \bar{v}^\alpha}{p} \right)\left(\mfrac{-\lambda}{p}\right)^{\alpha}=\left( \mfrac{(-1)^{\delta(\alpha+1)} s^\delta (\lambda\bar{v})^\alpha}{p} \right).
\end{equation}

\noindent{\bf Case 1.} Suppose that $\lambda=2$, $\alpha$ is odd and $(p, 
\alpha) \neq (3, 1)$. By Lemma~\ref{lemma v mod p}(i),
$2\bar{v}p^{\delta} = s(sp^{\alpha}-1)p^{\alpha}$. So, since $s(sp^{\alpha}-1)
\equiv -s \mod{p}$, it follows that $\delta=\alpha$ and $2\bar{v} \equiv -s
\mod{p}$. Then \eqref{eq no hamilton cycle v odd}
implies $C_{p}(X)=(\frac{-s^{2}}{p})=(\frac{-1}{p})=-1$ as required.

\noindent{\bf Case 2.} Suppose that $\lambda>2$. Let $\ell p^{\gamma}$ be the $p$-factorisation of $\lambda - 2$. Because $v=\bar{v}p^\delta$ and $\lambda=\ell p^{\gamma}+2$, Lemma \ref{lemma v mod p}(i) implies that
\begin{align}
\lambda\bar{v}p^\delta  &= (\ell p^{\gamma})^2 + sp^{\alpha}(2(\ell p^{\gamma}+2) -5 + sp^{\alpha}) \nonumber\\
\lambda\bar{v}p^{\delta-\alpha} &= \ell^2p^{2\gamma-\alpha} + s(2\ell p^{\gamma}  + sp^{\alpha} -1). \label{eqn no ham cycle case 2 eq}
\end{align}
Recall that $\alpha \geq 1$ and note that $\lambda \equiv 2 \mod {p}$ by Lemma
\ref{lemma v mod p}(ii), so $\gamma \geq 1$. The proof now splits into subcases
according to whether the assumptions of (i) or (ii) hold.

\noindent{\bf Case 2a.} Suppose further that $\alpha$ is odd, $\alpha < 2\gamma$ and $(p, \alpha) \neq (3, 1)$.  Then the right hand side of \eqref{eqn no ham cycle case 2 eq} is an integer congruent to $-s$ modulo $p$. So $\delta=\alpha$ and $\lambda\bar{v} \equiv -s \mod{p}$. Now, using \eqref{eq no hamilton cycle v odd},
we have $C_{p}(X) = (\frac{-s^2}{p}) = (\frac{-1}{p}) = -1$ as
required.

\noindent{\bf Case 2b.} Suppose further that $\alpha = 2\gamma$ and $\delta$ is odd. Then the right hand side of \eqref{eqn no ham cycle case 2 eq} is an integer and, because $p$ does not divide $\lambda$, it follows that $\delta \geq \alpha+1$. So $p$ divides the right hand side of \eqref{eqn no ham cycle case 2 eq} and it follows that $s \equiv \ell^2 \mod{p}$. Now, using \eqref{eq no hamilton cycle v odd},
we have $C_{p}(X) = (\frac{-s}{p}) = (\frac{-\ell^2}{p}) = -1$ as
required.
\end{proof}

\begin{remark}
It can be shown that Theorem~\ref{theorem: no hamilton cycle v odd} is close to
the best result achievable via Lemma~\ref{lemma: main lemma}. Specifically, if
$k>\lambda+2$, $k-\lambda+2 \equiv 0 \mod{p}$, $v = \frac{k(k-1)-2}{\lambda}+1$
and $|X_{(v,k,\lambda)}[v]|$ is a perfect square, but the hypotheses of Theorem~\ref{theorem:
no hamilton cycle v odd} do not hold (because $p \not\equiv 3 \mod{4}$ or $v
\not\equiv p \mod{2p}$ or because (i) and (ii) fail), then $C_{p}(X_{(v,k,\lambda)}[v]) = 1$
unless $(p, \alpha) = (3, 1)$ and $s \equiv 1 \mod{3}$. When $(p, \alpha) = (3,
1)$ and $s \equiv 1 \mod{3}$, we have $C_3(X_{(v,k,\lambda)}[v])=-1$ for some $v$ and
$C_3(X_{(v,k,\lambda)}[v])=1$ for other $v$. In the interests of brevity we do not prove any of this here, however.
\end{remark}

We give an example of an infinite family of parameter sets for which Theorem 
\ref{theorem: no hamilton cycle v odd} rules out the existence of a symmetric 
covering with  Hamilton cycle excess.

\begin{corollary} \label{c: infinite family Hamilton}
Suppose $p$ is an odd prime, where $p \equiv 3 \mod{4}$, and $\alpha$ is an odd 
positive integer, such that $(p, \alpha ) \not= (3,1)$.  Then there does not 
exist a symmetric $(\frac{1}{2}p^\alpha(p^\alpha-1),p^\alpha,2)$-covering with  
Hamilton cycle excess.
\end{corollary}

\section{Excesses composed of uniform length cycles} \label{sec: uniform length cycles}

In this section we focus on establishing the nonexistence of symmetric
$(v,k,\lambda)$-coverings with excesses consisting of a number of cycles of the
same length. We begin with some computational results. Table~\ref{table: same cycle length} lists cycle types of the form $[n^t]$ that can be ruled out by Lemma~\ref{lemma: main lemma} as excesses of symmetric $(v,k,\lambda)$-coverings for $p < 10^4$, $1 \leq \lambda \leq 5$ and $\lambda + 2 < k < 30$. It does not include cycle types ruled out by Theorem~\ref{thm: BBHMS BRC} or those of the form $[v]$ (the latter are listed in Table~\ref{table: Ham cycles
excess}).

\begin{table}[h!]
\begin{center}
\footnotesize
\begin{tabular}{|lll|l|l||lll|l|l|}
\hline
 $v$ & $k$ & $\lambda$ & $[n^t]$ & $p$ & $v$ & $k$ & $\lambda$ & $[n^t]$ & $p$\\\hline\hline
55 & 8 & 1 & $ 11 ^{ 5 }$ & 43, 307 & 253 & 23 & 2 & $ 11 ^{ 23 }$ & 43 \\
155 & 13 & 1 & $ 31 ^{ 5 }$ & 2, \textbf{7} &  &  &  & $ 23 ^{ 11 }$ & \textbf{23} \\
305 & 18 & 1 & $ 61 ^{ 5 }$ & 6709 & 300 & 25 & 2 & $ 2 ^{ 150 }$ & 3, 7 \\
341 & 19 & 1 & $ 31 ^{ 11 }$ & 557, 2417 &  &  &  & $ 6 ^{ 50 }$ & 3, 7 \\
505 & 23 & 1 & $ 5 ^{ 101 }$ & 2, \textbf{3} &  &  &  & $ 10 ^{ 30 }$ & 3, 7 \\
 &  &  & $ 101 ^{ 5 }$ & 2, \textbf{3} &  &  &  & $ 30 ^{ 10 }$ & 3, 7 \\
&&&& &  &  &  & $ 50 ^{ 6 }$ & 3, 7 \\
15 & 6 & 2 & $ 3 ^{ 5 }$ & 2, 3 &  &  &  & $ 150 ^{ 2 }$ & 3, 7 \\
 &  & & $ 5 ^{ 3 }$ & 2, \textbf{3} & 325 & 26 & 2 & $ 5 ^{ 65 }$ & 19, 29 \\
21 & 7 & 2 & $ 7 ^{ 3 }$ & \textbf{7}, 13 &  &  &  & $ 25 ^{ 13 }$ & 2, \textbf{13}, 19, 29 \\
28 & 8 & 2 & $ 4 ^{ 7 }$ & 2, 3 &  &  &  & $ 65 ^{ 5 }$ & 2, \textbf{13}, 19, 29 \\
45 & 10 & 2 & $ 9 ^{ 5 }$ & 2, \textbf{5} & 351 & 27 & 2 & $ 3 ^{ 117 }$ & 2, \textbf{3} \\
 &  &  & $ 15 ^{ 3 }$ & 2, \textbf{5}, 29, 149 &  &  &  & $ 9 ^{ 39 }$ & 2, 71 \\
55 & 11 & 2 & $ 11 ^{ 5 }$ & \textbf{11}, 197 &  &  &  & $ 27 ^{ 13 }$ & 2, \textbf{3}, 71 \\
78 & 13 & 2 & $ 6 ^{ 13 }$ & 2, 5 &  &  &  & $ 39 ^{ 9 }$ & 2, 3 \\
91 & 14 & 2 & $ 7 ^{ 13 }$ & 2, 223 &  &  &  & $ 117 ^{ 3 }$ & 2, 71, 233, 1637 \\
105 & 15 & 2 & $ 15 ^{ 7 }$ & 3, \textbf{5}, 59, 509 & 406 & 29 & 2 & $ 14 ^{ 29 }$ & 41, 461 \\
&  & & $ 21 ^{ 5 }$ & 3, \textbf{5}, 419 & &&&& \\
&  &  & $ 35 ^{ 3 }$ & \textbf{3}, \textbf{5} & 169 & 23 & 3 & $ 13 ^{ 13 }$ & 2, \textbf{11} \\
120 & 16 & 2 & $ 4 ^{ 30 }$ & 2, 3 & &&&&\\
 & & & $ 12 ^{ 10 }$ & 2, 3 & 176 & 27 & 4 & $ 8 ^{ 22 }$ & 3, 7 \\
 & & & $ 20 ^{ 6 }$ & 2, 3 &  &  &  & $ 88 ^{ 2 }$ & 3, 7 \\
&  &  & $ 60 ^{ 2 }$ & 2, 3 & &&&&\\
153 & 18 & 2 & $ 3 ^{ 51 }$ & 2, 5 & 15 & 9 & 5 & $ 3 ^{ 5 }$ & 2, 3 \\
 &  & & $ 9 ^{ 17 }$ & 5, 71 &  &  &  & $ 5 ^{ 3 }$ & 2, 3 \\
 &  &  & $ 17 ^{ 9 }$ & 101 & 27 & 12 & 5 & $ 3 ^{ 9 }$ & 2, \textbf{3} \\
 & &  & $ 51 ^{ 3 }$ & 2, 5, 101, 2447, 5303 &  &  &  & $ 9 ^{ 3 }$ & 2, 107 \\
171 & 19 & 2 & $ 19 ^{ 9 }$ & \textbf{19}, 113, 227 & 55 & 17 & 5 & $ 5 ^{ 11 }$ & 2, 7 \\
 &  &  & $ 57 ^{ 3 }$ & \textbf{19}, 113, 227, 4217 &  &  &  & $ 11 ^{ 5 }$ & 2, 7 \\
190 & 20 & 2 & $ 38 ^{ 5 }$ & 37, 113, 797 & 93 & 22 & 5 & $ 31 ^{ 3 }$ & 991 \\
231 & 22 & 2 & $ 3 ^{ 77 }$ & 2, \textbf{11} & 111 & 24 & 5 & $ 3 ^{ 37 }$ & 2, 3 \\
 &  &  & $ 11 ^{ 21 }$ & 2 & 141 & 27 & 5 & $ 47 ^{ 3 }$ & 2, \textbf{3} \\
 &  &  & $ 21 ^{ 11 }$ & 2, \textbf{11}, 41 & &&&&\\
 &  &  & $ 33 ^{ 7 }$ & \textbf{11} & &&&&\\
 &  &  & $ 77 ^{ 3 }$ & 2 & &&&&\\
 \hline
\end{tabular}
\end{center}
\vspace{-0.6cm}

\caption{Cycle types $[n^t]$ that are ruled out by Lemma~\ref{lemma: main lemma} as excesses of symmetric $(v,k,\lambda)$-coverings.}
\label{table: same cycle length}
\end{table}

As in the previous section, we note that when $1 \leq \lambda \leq 5$ and
$\lambda + 2 < k < 30$, more cases can be ruled out using Lemma~\ref{lemma:
main lemma} when $v$ is odd than when $v$ is even. Furthermore,
Theorem~\ref{thm: BBHMS BRC} has already ruled out a significant portion of the
cases
when $v$ is even but none of the cases when $v$ is odd. Consequently we
investigate the case in which $v$ is odd, and hence both the number of cycles in the excess
and the cycle length are odd. Theorem~\ref{theorem same cycle length}
treats choices of $p$ that do not divide the cycle length and
Theorem~\ref{theorem same
cycle length divisible} treats choices of $p$ that do.  In Table~\ref{table:
same cycle
length}, we mark in boldface the choices of $p$ for which Theorem~\ref{theorem
same
cycle length}~or~\ref{theorem same cycle length divisible} can be used to rule
out the case.

\begin{theorem} \label{theorem same cycle length}
Let $n$, $t$, $k$ and $\lambda$ be positive integers such that $k > \lambda + 2$, $nt = \frac{k(k-1)-2}{\lambda}+1$ and $nt$ is odd. There does not exist a symmetric $(nt,k,\lambda)$-covering whose excess consists of $t$ cycles of length $n$ if there is an odd prime $p$ such that
$k-\lambda+2 \equiv 0 \mod{p}$, $n \not\equiv 0 \mod{p}$ and
\begin{itemize}
	\item $\alpha$ is even, $\gamma$ is odd and $(\frac{s}{p})=-1$; or
	\item $\alpha$ is odd, $\gamma$ is even and $\left(\mfrac{(-1)^{(t-1)/2}n\bar\lambda}{p}\right)=-1$; or
	\item $\alpha$ is odd, $\gamma$ is odd and $\left(\mfrac{(-1)^{(t+1)/2}ns\bar\lambda}{p}\right)=-1$;
\end{itemize}
where $sp^\alpha$ and $\bar{\lambda}p^{\gamma}$ are the $p$-factorisations of $k-\lambda+2$ and $\lambda$ respectively. Furthermore, for any odd prime $p$ such that $k-\lambda+2 \equiv 0 \mod{p}$ and $n \not\equiv 0 \mod{p}$ but $p$ does not satisfy the above hypotheses, $C_p(X_{(nt,k,\lambda)}[n^t]) = 1$.
\end{theorem}

\begin{proof}
Let $p$ be a prime satisfying the hypotheses of the theorem. Let $X=X_{(nt,k,\lambda)}[n^t]$. We may assume that $|X|$ is a perfect square for otherwise we are finished by Proposition \ref{prop: matrix form}. By Lemma \ref{lem: full general cp}, noting that $n$ and $t$ are odd, we have
\begin{equation}\label{eqn same cycle length}
C_{p}(X) = C_{p}(B_n(k-\lambda)) (k-\lambda+2, -1)^{(t-1)/2}_{p} (-\lambda,
k-\lambda+2)_{p}.
\end{equation}
Since $n \not\equiv 0 \mod p$ and $n$ is odd, Lemma \ref{lemma lambda 2 Ham cycle}(ii) implies that $C_p(B_n(k-\lambda)) = (\frac{-n}{p})^\alpha$.
Also, using \eqref{eqn: Hil Legendre p}, $(-\lambda, k-\lambda+2)_p =
(- \bar{\lambda}p^{\gamma} , sp^\alpha)_p = (\frac{-1}{p})^{\alpha\gamma}(
\frac{-\bar\lambda}{p})^\alpha (\frac{s}{p})^\gamma$ and $(k-\lambda+2, -1)^{(t-1)/2}_p = (sp^\alpha, -1)^{(t-1)/2}_p = (\frac{-1}{p})^{\alpha(t-1)/2}$.
So, from \eqref{eqn same cycle length}, we have
$$C_p(X) = \left(
\mfrac{(-1)^{\alpha(\gamma + (t-1)/2)} n^\alpha \bar\lambda^\alpha s^\gamma}{p} \right).$$
The result now follows from Lemma \ref{lemma: main lemma} by checking cases.
\end{proof}

We remarked after Theorem~\ref{theorem: no hamilton cycle v odd} that Lemma~\ref{lemma: main lemma} cannot rule out Hamilton cycle excesses when $v \not\equiv 0 \mod{p}$.  It follows that Theorem~\ref{theorem same cycle length} never rules out Hamilton cycle excesses.

The following corollary gives one example of an infinite family of symmetric coverings with specified excesses whose existence is ruled out by Theorem \ref{theorem same cycle length}.

\begin{corollary} \label{c: infinite family same cycle not div by p}
Suppose $p$ is prime and $p \equiv 3, 5 \mod{8}$. Then there does not exist a 
symmetric $(2p^2-p,2p,2)$-covering with excess consisting of $p$ cycles of 
length $2p-1$.
\end{corollary}

\begin{theorem} \label{theorem same cycle length divisible}
Let $n$, $t$, $k$ and $\lambda$ be positive integers such that $k > \lambda +
2$, $nt = \frac{k(k-1)-2}{\lambda}+1$ and $nt$ is odd. There does not exist a symmetric
$(nt,k,\lambda)$-covering whose excess consists of $t$ cycles of length $n$ if
there is an odd prime $p$ such that $k-\lambda+2 \equiv 0 \mod{p}$, $n \equiv 0
\mod{p}$, $k-\lambda+2 \equiv 0 \mod{9}$ if $p=3$, and
\begin{itemize}
	\item $\alpha$ is even, $\delta$ is odd and $(\frac{-s}{p})=-1$; or
	\item $\alpha$ is odd, $\delta$ is even and $\left(\mfrac{(-1)^{(t-1)/2}2\bar{n}}{p}\right)=-1$; or
	\item $\alpha$ is odd, $\delta$ is odd and $\left(\mfrac{(-1)^{(t-1)/2}2s\bar{n}}{p}\right)=-1$;
\end{itemize}
where $sp^\alpha$ and $\bar{n}p^{\delta}$ are the $p$-factorisations of $k-\lambda+2$ and $n$ respectively. Furthermore, for any odd prime $p$ such that $k-\lambda+2 \equiv 0 \mod{p}$, $n \equiv 0 \mod{p}$, and $k-\lambda+2 \equiv 0 \mod{9}$ if $p=3$, but $p$ does not satisfy the above hypotheses, $C_p(X_{(nt,k,\lambda)}[n^t]) = 1$.
\end{theorem}

\begin{proof}
Let $p$ be a prime satisfying the hypotheses of the theorem. Let $X=X_{(nt,k,\lambda)}[n^t]$. We may assume that $|X|$ is a perfect square for otherwise we are finished by Proposition \ref{prop: matrix form}. By Lemma \ref{lem: full general cp}, noting that $n$ and $t$ are odd, we have
\begin{equation}\label{eqn same cycle length v2}
C_{p}(X) = C_{p}(B_n(k-\lambda)) (k-\lambda+2, -1)^{(t-1)/2}_{p} (-\lambda,
k-\lambda+2)_{p}.
\end{equation}
Because $n$ is odd, $g_n(a) = \bar{g}p^{\delta}$ for some integer $\bar{g} \equiv \bar{n} \mod {p}$ by Lemma \ref{lemma g star deltas} and hence Lemma \ref{lemma lambda 2 Ham cycle}(i) implies that
$$C_p(B_n(k-\lambda)) =  \left( \mfrac{(-1)^{(\alpha\delta + \alpha +
\delta)}s^\delta \bar{n}^\alpha}{p} \right).$$
By Lemma~\ref{lemma v mod p}(ii), $\lambda \equiv 2
\mod{p}$. So, using \eqref{eqn: Hil Legendre p}, $(-\lambda, k-\lambda+2)_p =
(-\lambda , sp^\alpha)_p = (\frac{-\lambda}{p})^{\alpha} = (\frac{-2}{p})^{\alpha}$ and $(k-\lambda+2, -1)^{(t-1)/2}_p = (sp^\alpha, -1)^{(t-1)/2}_p = (\frac{-1}{p})^{\alpha(t-1)/2}$ .
Thus, from \eqref{eqn same cycle length v2}, we have
$$C_p(X) = \left( \mfrac{ (-1)^{\alpha(t-1)/2+\alpha\delta+\delta}s^\delta 2^\alpha\bar{n}^\alpha}{p} \right).$$
The result now follows from Lemma \ref{lemma: main lemma} by checking cases.
\end{proof}

Theorem~\ref{theorem same cycle length divisible} with $t = 1$ produces identical results to Theorem~\ref{theorem: no hamilton cycle v odd}.  However, we were able to phrase Theorem~\ref{theorem: no hamilton cycle v odd} without resorting to Legendre symbols.

Again, we give an example of an infinite family of symmetric coverings with specified excesses whose existence is ruled out by Theorem \ref{theorem same cycle length divisible}.

\begin{corollary} \label{c: infinite family same cycle div by p}
Suppose $p$ is prime, $p > 3$, and $p \equiv 3, 7 \mod{8}$. Then there does not 
exist a symmetric $(\frac{1}{2}p(p-1), p ,2)$-covering with excess consisting 
of $\frac{p-1}{2}$ cycles of length $p$.
\end{corollary}

\ignore{
We conclude this section with some computational results obtained by applying both Theorems ~\ref{theorem same cycle length} and ~\ref{theorem same cycle length divisible}. In the following discussion, a case is a triple $(nt, k, \lambda)$, where
$nt = \frac{k(k-1)-2}{\lambda} +1$ and $t \geq 2$. The cases in
Table~\ref{table: same cycle length} which are covered by Theorems~
\ref{theorem same cycle length}~and~\ref{theorem same cycle length divisible}
are noted in boldface and satisfy $\lambda \equiv 0 \mod{p}$ or $n \equiv 0
\mod{p}$, respectively.  We remark that out of all possible cases when $1
\le \lambda \le 10$ and $\lambda+2 < k < 1000$ and $p < 10^4$,
Theorems~\ref{theorem same cycle length}~and~\ref{theorem same cycle length
divisible} together rule out between 40 and 80 percent. Moreover, none of these
cases are ruled out by Theorem~\ref{thm: BBHMS BRC}.

\begin{center}
\footnotesize
\begin{tabular}{l||c|c|c|c|c|c|c|c|c|c}
$\lambda$  &  1  & 2 &  3  & 4 & 5 & 6 & 7 & 8 & 9 & 10 \\
\hline
\% cases ruled out by Theorem \ref{theorem same cycle length} or \ref{theorem same cycle length divisible}  &
58.8  & 82.2 & 44.0 & 61.6 & 53.3 & 56.3 & 62.2 & 47.8 & 62.0 & 63.1 \\
\end{tabular}
\end{center}

}

\section{Excesses composed of 2-cycles and 3-cycles} \label{sec: 2 and 3 cycles}

In this section we focus on establishing the nonexistence of symmetric $(v,k,\lambda)$-coverings whose excesses consist of 2-cycles and 3-cycles. As mentioned, results of Bose and Connor (see \cite{BC52}) already cover the cases in which the excess is composed entirely of $2$-cycles or entirely of $3$-cycles.

Table~\ref{table: 2 and 3 cycles} lists cycle types of the form $[2^{t_2}, 3^{t_3}]$ that can be ruled out by Lemma~\ref{lemma: main lemma} as excesses of symmetric $(v,k,1)$-coverings for $p < 10$ and $4 \leq k \leq 10$. Computational results for small values of $\lambda$ and $k$ show that taking $p = 5$ or $p = 2$ often rules out cycle types composed entirely of $2$-cycles and $3$-cycles. In Theorem \ref{theorem 2s and 3s p5} and Lemma \ref{theorem 2s and 3s} we consider the choices $p=5$ and $p=2$ respectively. In Table \ref{table: 2 and 3 cycles}, we mark in boldface the cases  for which $p = 2$ rules out the case using Theorem \ref{theorem 2s and 3s} and for which $p = 5$ rules out the case using Theorem \ref{theorem 2s and 3s p5}.

\begin{table}[h!]
\begin{center}
\footnotesize
\begin{tabular}{|lll|l|l|}
\hline
 $v$ & $k$ & $\lambda$ & $[2^{t_2}, 3^{t_3}]$ & $p$ \\\hline\hline
11 & 4 & 1  & $2^{ 4 }3^{ 1 }$ & \textbf{2}, \textbf{5}\\
\hline
19 & 5 & 1  & $2^{ 2 }3^{ 5 }$ & \textbf{2}, 3\\
& & & $2^{ 5 }3^{ 3 }$ & 2, 3\\
\hline
29 & 6 & 1  & $2^{ 1 }3^{ 9 }$, $2^{ 13 }3^{ 1 }$ & 2, 3\\
& & & $2^{ 7 }3^{ 5 }$ & 2, 7\\
& & & $2^{ 10 }3^{ 3 }$ & 3, 7\\
\hline
41 & 7 & 1  & $2^{ 1 }3^{ 13 }$ , $2^{ 4 }3^{ 11 }$ , $2^{ 7 }3^{ 9 }$ , $2^{ 10 }3^{ 7 }$ , $2^{ 13 }3^{ 5 }$, $2^{ 16 }3^{ 3 }$ , $2^{ 19 }3^{ 1 }$ & 2, \textbf{5}\\
\hline
55 & 8 & 1  & $2^{ 2 }3^{ 17 }$, $2^{ 8 }3^{ 13 }$, $2^{ 14 }3^{ 9 }$, $2^{ 20 }3^{ 5 }$, $2^{ 26 }3^{ 1 }$  & \textbf{2}, 3\\
& & & $2^{ 5 }3^{ 15 }$, $2^{ 11 }3^{ 11 }$ , $2^{ 17 }3^{ 7 }$, $2^{ 23 }3^{ 3 }$ & 3, \textbf{5}\\
\hline
71 & 9 & 1  & $2^{ 1 }3^{ 23 }$, $2^{ 13 }3^{ 15 }$, $2^{ 25 }3^{ 7 }$ & 2, 3\\
& & & $2^{ 4 }3^{ 21 }$, $2^{ 16 }3^{ 13 }$, $2^{ 28 }3^{ 5 }$ & \textbf{2}, \textbf{5} \\
& & & $2^{ 10 }3^{ 17 }$, $2^{ 22 }3^{ 9 }$, $2^{ 34 }3^{ 1 }$  & 3, \textbf{5} \\
\hline
89 & 10 & 1  & $2^{ 4 }3^{ 27 }$, $2^{ 16 }3^{ 19 }$, $2^{ 28 }3^{ 11 }$, $2^{ 40 }3^{ 3 }$ & 2\\
& & & $2^{ 7 }3^{ 25 }$, $2^{ 19 }3^{ 17 }$, $2^{ 31 }3^{ 9 }$, $2^{ 43 }3^{ 1 }$  & 7\\
& & & $2^{ 10 }3^{ 23 }$ , $2^{ 22 }3^{ 15 }$ , $2^{ 34 }3^{ 7 }$ & 2, 7\\
 \hline
\end{tabular}
\end{center}
\vspace{-0.6cm}

\caption{Cycle types of the form $[2^{t_2}, 3^{t_3}]$ that are ruled out by Lemma~\ref{lemma: main lemma} as excesses of symmetric $(v,k,1)$-coverings.}
\label{table: 2 and 3 cycles}
\end{table}

Lemma \ref{lemma: B3'} gives a concise expression for $C_p(B_3(a))$. We use this to prove Theorem \ref{theorem 2s and 3s p5} and Theorem \ref{theorem 2s and 3s}.

\begin{lemma} \label{lemma: B3'}
Let $a \geq 2$ be a positive integer and let $p$ be a prime. Then $C_p(B_3(a)) = (-1,-1)_p (-a-2,a-1)_p$.
\end{lemma}

\begin{proof}
Let $Y'$ be the matrix $\diag(B_3,-1)$. Note that $|B_3|=(a+2)(a-1)^2$. Then
\begin{align*}
C_p(B_3)  &= C_p(Y')(|Y'|, -|B_3|)_p && \hbox{by rearranging \eqref{eqn: cp
recursion}} \\
&= C_p(Y')(-|B_3|, -|B_3|)_p  && \hbox{since $|Y'|=-|B_3|$} \\
&= C_p(Y')(-a-2,-1)_p && \hbox{by \eqref{eqn: Hil a a} since $|B_3|=(a+2)(a-1)^2$}.
\end{align*}
Let $Y''$ be the matrix obtained from $Y'$ by adding the last row to all other
rows and then adding the last column to all other columns. Note that the $3$rd
principal minor of $Y''$ is $(a-1)I_3$ and that by applying~\eqref{eqn: cp
block diag} twice, we get that $C_p((a-1)I_3)=(-1,-1)_p$. Using the equation
above, we have
\begin{align*}
C_p(B_3) &= C_p(Y'')(-a-2,-1)_p && \hbox{since $Y'' \sim Y'$} \\
 &= C_p((a-1)I_3) (|Y''|, 1-a)_p(-a-2,-1)_p  && \hbox{by \eqref{eqn: cp recursion} since $|(a-1)I_3|=(a-1)^3$} \\
 &= C_p((a-1)I_3) (-a-2,1-a)_p(-a-2,-1)_p  && \hbox{since $|Y''| = |Y'| = -(a+2)(a-1)^2$} \\
 &= (-1,-1)_p (-a-2,1-a)_p(-a-2,-1)_p  && \hbox{since $C_p((a-1)I_3)=(-1,-1)_p$} \\
 &= (-1,-1)_p(-a-2,a-1)_p && \hbox{by \eqref{eqn: Hil split}.}
\end{align*}
\end{proof}

\begin{theorem} \label{theorem 2s and 3s p5}
Let $t_2$, $t_3$, $\lambda$ and $k$ be positive integers such that $k- \lambda > 2$, $\lambda$ is not divisible by $5$ and $2t_2+3t_3 = \frac{k(k-1)-2}{\lambda}+1$. There does not exist a symmetric $(2t_2+3t_3,k, \lambda)$-covering
whose excess consists of $t_2$ cycles of length $2$ and $t_3$ cycles of length
$3$ if
\begin{itemize}
	\item[(i)] $k - \lambda = 5^\alpha s+1$ where $\alpha$ is odd, $s 
\not\equiv 0 \mod{5}$, and $t_3$ is odd; or
	\item[(ii)] $k - \lambda = 5^\alpha s+2$ where $\alpha$ is odd, $s 
\not\equiv 0 \mod{5}$, $t_2$ is odd and $\lambda \equiv 1,4 \mod{5}$; or
	\item[(iii)] $k - \lambda = 5^\alpha s-2$ where $\alpha$ is odd, $s 
\not\equiv 0 \mod{5}$, $t_2+t_3$ is odd and $\lambda \equiv 1,4 \mod{5}$.
\end{itemize}
\end{theorem}
\begin{proof}
Suppose that one of (i), (ii) or (iii) holds. Let $X=X_{(2t_2+3t_3,k,1)}[2^{t_2},3^{t_3}]$. We may assume that $|X|$ is a perfect square for otherwise we are finished by Proposition \ref{prop: matrix form}. Let $a = k - \lambda$. In the rest of the proof we often use the fact that $-1 \equiv 2^2 \mod{5}$.  By Lemma \ref{lem: full general cp},
\begin{align*}
f_5(a,\lambda,t_2+t_3,t_2) &= (a+2,a^2-4)^{t_2t_3}_5(-\lambda,(a+2)^{t_2+t_3}
(a-2)^{t_2})_5\\
&= (-\lambda,(a+2)^{t_2+t_3} (a-2)^{t_2})_5,
\end{align*}
where the last equality follows because $(a+2, a^2-4)_5 = (a+2, -1)_5(a+2,a-2)_5 = 1$, which is derived using \eqref{eqn: Hil split} and \eqref{eqn: Hil a a} and by checking each equivalence class of $a$ modulo $5$.

Observe that $C_p(B_2(a)) = (-a, 4-a^2)_p$. By Lemma \ref{lem: full general cp} and Lemma \ref{lemma: B3'},
\begin{align*}
C_5(X) &= (-\lambda,(a+2)^{t_2+t_3} (a-2)^{t_2})_5C_5(B_2(a))^{t_2} C_5(B_3(a))^{t_3} \\
&= (-\lambda, (a+2)^{t_2+t_3} (a-2)^{t_2})_5(-a, 4-a^2)_5^{t_2} (-a-2, a-1)_5^{t_3}.
\end{align*}
It is routine to check that, when one of
(i), (ii) or (iii) holds, $C_5(X) = -1$ and the result follows from Lemma
\ref{lemma: main lemma}.
\end{proof}

If none of (i), (ii) or (iii) holds, then $C_5(X) = 1$ and Lemma \ref{lemma: 
main lemma} cannot be used to rule out the existence of a 
$(2t_2+3t_3,k,1)$-covering with  excess having cycle type $[2^{t_2}, 3^{t_3}]$.

Observe that Theorem \ref{theorem 2s and 3s p5} rules out every cycle type $[2^{t_2}, 3^{t_3}]$ as a possible excess for a symmetric $(41, 7, 1)$-covering (see Table \ref{table: 2 and 3 cycles}).  This generalises to a direct corollary of Theorem \ref{theorem 2s and 3s p5}(i), which rules out any excess of cycle type $[2^{t_2}, 3^{t_3}]$ for an infinite family of symmetric $(v,k,1)$-coverings.

\begin{corollary} \label{c: infinite family 2s 3s prime5}
If $v$ and $k$ are positive integers such that $k \equiv 7,12,17,22 \mod {25}$ 
and $v  = k(k-1)-1$, then there does not exist a symmetric $(v, k, 1)$-covering 
with  excess consisting of $2$- and $3$-cycles.
\end{corollary}

Observing the examples in Table~\ref{table: 2 and 3 cycles}, we see that
a symmetric $(55,8,1)$-covering also cannot have excess consisting only of $2$- and $3$-cycles. However, Theorem~\ref{theorem 2s and 3s p5}(ii) establishes this only when there is an odd number of cycles of length $2$. To rule out the excess types with even number of $2$-cycles, we employ Lemma \ref{lemma: main lemma} with $p=2$ in Lemma \ref{theorem 2s and 3s} below. This enables us to give another infinite family of parameters for which there does not exist a symmetric covering with excess having only $2$- and $3$-cycles.

\begin{lemma} \label{theorem 2s and 3s}
Let $t_2$, $t_3$ and $k$ be positive integers such that $k > 3$
and $2t_2+3t_3 = k(k-1)-1$. There does not exist a symmetric
$(2t_2+3t_3,k,1)$-covering whose excess consists of $t_2$ cycles of length $2$
and $t_3$ cycles of length $3$ if
\begin{itemize}
	\item[(i)] $k \equiv 0 \mod{4}$ and $t_3\equiv 1 \mod{4}$; or
	\item[(ii)] $k \equiv 1 \mod{4}$ and $t_3\equiv 5 \mod{8}$.
\end{itemize}
\end{lemma}

\begin{proof}
Let $X=X_{(2t_2+3t_3,k,1)}[2^{t_2},3^{t_3}]$. We may assume that $|X|$ is a
perfect square for otherwise we are finished by Proposition \ref{prop: matrix
form}. By Lemma \ref{lemma: main lemma}, it suffices to establish that
$C_{2}(X)=1$. Suppose that (i) or (ii) holds. Then $t_3 \equiv 1 \mod{4}$, $2t_2+3t_3=
k(k-1)-1 \equiv 3 \mod{4}$ and so $t_2$ is even. Thus, by
Lemma \ref{lem: full general cp}, for any prime $p$,
\begin{align*}
C_{p}(X) &= (k^2-2k-3, -1)_p^{t_2/2}(-1,k+1)_{p} C_{p}(B_3(k-1))
\\
&= (k^2-2k-3, -1)_p^{t_2/2}(-1,k+1)_{p} (-1,-1)_p(-k-1,k-2)_p&& \hbox{by Lemma
\ref{lemma: B3'}} \\
&= (k^2-2k-3, -1)_p^{t_2/2} (-k-1,-1)_p(-k-1,k-2)_p && \hbox{by \eqref{eqn: Hil
split}} \\
&= (k^2-2k-3, -1)_p^{t_2/2} (-k-1,-k+2)_p&& \hbox{by \eqref{eqn: Hil split}.}
\end{align*}
We can now establish that $C_{2}(X)=1$ by specialising this
equation to the case $p=2$ and applying \eqref{eqn: Hil Legendre 2}, considering
the cases $k \equiv 0 \mod{4}$, $k \equiv 1 \mod{8}$ and $k \equiv 5 \mod{8}$
separately. In the first case, $(-k-1,-k+2)_2=(k^2-2k-3, -1)_2=1$.
In the second case $(-k-1,-k+2)_2=1$ and it follows from $2t_2+3t_3=k(k-1)-1$ that $t_2 \equiv 0 \mod{4}$. In the third case $(-k-1,-k+2)_2=(k^2-2k-3, -1)_2=-1$ and it follows from $2t_2+3t_3= k(k-1)-1$ that $t_2 \equiv 2 \mod{4}$.
\end{proof}

Observe that under the hypotheses of Lemma~\ref{theorem 2s and 3s} if $k
\equiv 1 \mod{4}$ but $t_3 \equiv 1 \mod{8}$ then $C_2(X[2^{t_2},3^{t_3}]) = -1$
and Lemma \ref{lemma: main lemma} cannot be used to rule out the existence of a
$(2t_2+3t_3,k,1)$-covering with such an excess.  We also remark that the proof
of Lemma~\ref{theorem 2s and 3s} easily extends to rule out the existence of a
$(v,k,1)$-covering with excess having $v$-feasible cycle type $[2^{t_2},
3^{t_3}, c_1^{m_1}, \dots, c_{t}^{m_t}]$ where $m_i \equiv 0 \mod{4}$ for all $1
\le i \le t$.

In Lemma~\ref{theorem 2s and 3s}, if $k \equiv 0,1
\mod{4}$, then $k(k-1)-1$ is odd and hence $t_3$ is odd. Therefore, for a fixed 
$k$, parts (i) and (ii) of the Lemma~\ref{theorem 2s and 3s} rule out, 
respectively, about a half and a quarter of the feasible cycle types of the form 
$[2^{t_2}, 3^{t_3}]$.

The following corollary is a straightforward application of
Theorem~\ref{theorem 2s and 3s p5}(ii) and Lemma~\ref{theorem 2s and 3s}(i)
(note that Lemma~\ref{theorem 2s and 3s}(i) applies whenever $k \equiv 0 \mod{4}$ and $t_2$ is even).

\begin{corollary}  \label{c: infinite family 2s 3s primes 2 5}
If $k$ is a positive integer such that $k \equiv 8, 48, 68, 88 \mod{100}$ and $v  = k(k-1)-1$, then there does not exist a symmetric $(v, k, 1)$-covering with excess consisting of 2- and 3-cycles.
\end{corollary}

\section{Conclusion}

In Sections~\ref{sec: p divides a}--\ref{sec: 2 and 3 cycles}, we ruled out the existence of several infinite families of symmetric $(v,k,\lambda)$-coverings with particular types of excesses using Lemma~\ref{lemma: main lemma}.
Observe that Theorem~\ref{thm: BBHMS BRC} rules out the existence of infinitely many symmetric $(v,k,\lambda)$-coverings with $2$-regular excess when $v$ is even. However, when $v$ is odd and $k > \lambda +2$, the following problem remains open.

\begin{open}
Are there infinitely many parameter sets $(v,k,\lambda)$ where 
$3 \leq \lambda+2 < k < v$, $v = \frac{k(k-1)-2}{\lambda}+1$ and $v$ is 
odd, for which there is no symmetric 
$(v,k,\lambda)$-covering?
\end{open}

Since cyclic symmetric coverings are of particular interest and have applications in related fields of study, as mentioned in Section~\ref{sec: computational results}, we note that the following problem remains open as well.

\begin{open}
Are there infinitely many parameter sets $(v,k,\lambda)$ where 
$3 \leq \lambda+2 < k < v$, $v = \frac{k(k-1)-2}{\lambda}+1$ and $v$ is 
odd, for which there is no cyclic symmetric 
$(v,k,\lambda)$-covering?
\end{open}

Obviously, an affirmative answer to the first question would answer both 
questions in the affirmative and a negative answer to the second question would 
answer both questions in the negative.

\vspace{0.3cm} \noindent{\bf Acknowledgements}

The authors were supported by Australian Research Council grants DE120100040 and
DP150100506.

\end{document}